\documentclass{article}
\usepackage{amsmath}
\usepackage{amssymb}
\usepackage{amsthm}
\usepackage{cite}
\usepackage{stmaryrd}
\usepackage{slashed} 
\usepackage{euscript}
\usepackage[arrow,matrix]{xy}
\usepackage[utf8]{inputenc}
\usepackage[unicode]{hyperref}
\usepackage{mathtools}

\usepackage{tikz}
\usetikzlibrary{matrix,positioning,arrows,decorations.pathreplacing}

\DeclareMathOperator\spec{spec}
\DeclareMathOperator\stab{stab}

\def\Fred{\mathfrak{Fred}}
\DeclareMathOperator\ind{ind}

\def\orient{\operatorname{or}}
\def\spectral{\operatorname{sp}}
\def\pfaff{\operatorname{pf}}
\def\Pfaff{\operatorname{Pf}}
\def\general{\operatorname{\lambda}}

\DeclareMathOperator\Ex{Ex}

\begin{document}
%%%%%%%%%%%%%%%%%%%%%%%%%%%%%%%%%%%%%%%%%%%%%%%%%%%%%%%%%%%%%%%%%%%%%%%%
%%%%%%%%%%%%%%%%%%%%%%%%%%     Macros      %%%%%%%%%%%%%%%%%%%%%%%%%%%%%
%%%%%%%%%%%%%%%%%%%%%%%%%%%%%%%%%%%%%%%%%%%%%%%%%%%%%%%%%%%%%%%%%%%%%%%%

\newcommand\Vtextvisiblespace[1][.3em]{%
  \mbox{\kern.06em\vrule height.3ex}%
  \vbox{\hrule width#1}%
  \hbox{\vrule height.3ex}}

\def\e#1\e{\begin{equation}#1\end{equation}}
\def\ea#1\ea{\begin{align}#1\end{align}}
\def\eq#1{{\rm(\ref{#1})}}
\theoremstyle{plain}% default
\newtheorem{thm}{Theorem}[section]
\newtheorem{lem}[thm]{Lemma}
\newtheorem{prop}[thm]{Proposition}
\newtheorem{cor}[thm]{Corollary}
\newtheorem{quest}[thm]{Question}
\newtheorem{prob}[thm]{Problem}
\theoremstyle{definition}
\newtheorem{dfn}[thm]{Definition}
\newtheorem{ex}[thm]{Example}
\newtheorem{rem}[thm]{Remark}
\newtheorem{ax}[thm]{Axiom}
\newtheorem{ass}[thm]{Assumption}
\newtheorem{property}[thm]{Property}
\newtheorem{cond}[thm]{Condition}

\numberwithin{figure}{section}
\numberwithin{equation}{section}
\def\dim{\mathop{\rm dim}\nolimits}
\def\codim{\mathop{\rm codim}\nolimits}
\def\vdim{\mathop{\rm vdim}\nolimits}
\def\sign{\mathop{\rm sign}\nolimits}
\def\Im{\mathop{\rm Im}\nolimits}
\def\det{\mathop{\rm Det}\nolimits}
\def\Ker{\mathop{\rm Ker}}
\def\Coker{\mathop{\rm Coker}}
\def\Spec{\mathop{\rm Spec}}
\def\Perf{\mathop{\rm Perf}}
\def\Vect{\mathop{\rm Vect}}
\def\LCon{\mathop{\rm LCon}}
\def\Flag{\mathop{\rm Flag}\nolimits}
\def\FlagSt{\mathop{\rm FlagSt}\nolimits}
\def\Iso{\mathop{\rm Iso}\nolimits}
\def\Aut{\mathop{\rm Aut}}
\def\End{\mathop{\rm End}\nolimits}

\def\Ho{\mathop{\rm Ho}}
\def\PGL{\mathop{\rm PGL}}
\def\GL{\mathop{\rm GL}}
\def\SL{\mathop{\rm SL}}
\def\SO{\mathop{\rm SO}}
\def\SU{\mathop{\rm SU}}
\def\Sp{\mathop{\rm Sp}}
\def\ch{\mathop{\rm ch}\nolimits}
\def\Spin{\mathop{\rm Spin}}
\def\SF{\mathop{\rm SF}}
\def\Tr{\mathop{\rm Tr}}
\def\U{{\mathbin{\rm U}}}
\def\vol{\mathop{\rm vol}}
\def\inc{\mathop{\rm inc}}

\def\tind{{\text{\rm t-ind}}}
\def\bdim{{\mathbin{\bf dim}\kern.1em}}
\def\rk{\mathop{\rm rk}}
\def\Pic{\mathop{\rm Pic}}
\def\colim{\mathop{\rm colim}\nolimits}
\def\Stab{\mathop{\rm Stab}\nolimits}
\def\Exact{\mathop{\rm Exact}\nolimits}
\def\Crit{\mathop{\rm Crit}}
\def\supp{\mathop{\rm supp}}
\def\rank{\mathop{\rm rank}\nolimits}
\def\Hom{\mathop{\rm Hom}\nolimits}
\def\bHom{\mathop{\bf Hom}\nolimits}
\def\Ext{\mathop{\rm Ext}\nolimits}
\def\cExt{\mathop{{\mathcal E}\mathit{xt}}\nolimits}
\def\id{{\mathop{\rm id}\nolimits}}
\def\Id{{\mathop{\rm Id}\nolimits}}
\def\Sch{\mathop{\bf Sch}\nolimits}
\def\Map{{\mathop{\rm Map}\nolimits}}
\def\AlgSp{\mathop{\bf AlgSp}\nolimits}
\def\Art{\mathop{\bf Art}\nolimits}
\def\HSt{\mathop{\bf HSt}\nolimits}
\def\dSt{\mathop{\bf dSt}\nolimits}
\def\dArt{\mathop{\bf dArt}\nolimits}
\def\TopSta{{\mathop{\bf TopSta}\nolimits}}
\def\Gpds{{\mathop{\bf Gpds}\nolimits}}
\def\Aff{\mathop{\bf Aff}\nolimits}
\def\SPr{\mathop{\bf SPr}\nolimits}
\def\SSet{\mathop{\bf SSet}\nolimits}
\def\Parf{\mathop{\bf Perf}}
\def\Ad{\mathop{\rm Ad}}
\def\pl{{\rm pl}}
\def\mix{{\rm mix}}
\def\fd{{\rm fd}}
\def\fpd{{\rm fpd}}
\def\pfd{{\rm pfd}}
\def\coa{{\rm coa}}
\def\rsi{{\rm si}}
\def\rst{{\rm st}}
\def\ss{{\rm ss}}
\def\vi{{\rm vi}}
\def\smq{{\rm smq}}
\def\rsm{{\rm sm}}
\def\rArt{{\rm Art}}
\def\po{{\rm po}}
\def\spo{{\rm spo}}
\def\Kur{{\rm Kur}}
\def\dcr{{\rm dcr}}
\def\top{{\rm top}}
\def\fc{{\rm fc}}
\def\cla{{\rm cla}}
\def\num{{\rm num}}
\def\irr{{\rm irr}}
\def\red{{\rm red}}
\def\sing{{\rm sing}}
\def\virt{{\rm virt}}
\def\qcoh{{\rm qcoh}}
\def\coh{{\rm coh}}
\def\lft{{\rm lft}}
\def\lfp{{\rm lfp}}
\def\cs{{\rm cs}}
\def\dR{{\rm dR}}
\def\Obj{{\rm Obj}}
\def\cdga{{\mathop{\bf cdga}\nolimits}}
\def\Rmod{\mathop{R\text{\rm -mod}}}
\def\Top{{\mathop{\bf Top}\nolimits}}
\def\modKQ{\mathop{\text{\rm mod-}\K Q}}
\def\modKQI{\text{\rm mod-$\K Q/I$}}
\def\modCQ{\mathop{\text{mod-}\C Q}}
\def\modCQI{\text{\rm mod-$\C Q/I$}}
\def\ul{\underline}
\def\bs{\mathbf}
\def\ge{\geqslant}
\def\le{\leqslant\nobreak}
\def\boo{{\mathbin{\mathbf 1}}}
\def\O{{\mathcal O}}
\def\bA{{\mathbin{\mathbb A}}}
\def\bG{{\mathbin{\mathbb G}}}
\def\bL{{\mathbin{\mathbb L}}}
\def\P{{\mathbin{\mathbb P}}}
\def\bT{{\mathbin{\mathbb T}}}
\def\H{{\mathbin{\mathbb H}}}
\def\K{{\mathbin{\mathbb K}}}
\def\R{{\mathbin{\mathbb R}}}
\def\Z{{\mathbin{\mathbb Z}}}
\def\Q{{\mathbin{\mathbb Q}}}
\def\N{{\mathbin{\mathbb N}}}
\def\C{{\mathbin{\mathbb C}}}
\def\CP{{\mathbin{\mathbb{CP}}}}
\def\KP{{\mathbin{\mathbb{KP}}}}
\def\RP{{\mathbin{\mathbb{RP}}}}
\def\fC{{\mathbin{\mathfrak C}\kern.05em}}
\def\fD{{\mathbin{\mathfrak D}}}
\def\fE{{\mathbin{\mathfrak E}}}
\def\fF{{\mathbin{\mathfrak F}}}
\def\A{{\mathbin{\cal A}}}
\def\G{{\mathbin{\cal G}}}
\def\M{{\mathbin{\cal M}}}
\def\uM{{\mathbin{\underline{\cal M\kern-.1em}\kern.1em}}}
\def\cB{{\mathbin{\cal B}}}
\def\cC{{\mathbin{\cal C}}}
\def\cD{{\mathbin{\cal D}}}
\def\cE{{\mathbin{\cal E}}}
\def\cF{{\mathbin{\cal F}}}
\def\cG{{\mathbin{\cal G}}}
\def\cH{{\mathbin{\cal H}}}
\def\E{{\mathbin{\cal E}}}
\def\F{{\mathbin{\cal F}}}
\def\cG{{\mathbin{\cal G}}}
\def\cH{{\mathbin{\cal H}}}
\def\cI{{\mathbin{\cal I}}}
\def\cJ{{\mathbin{\cal J}}}
\def\cK{{\mathbin{\cal K}}}
\def\cL{{\mathbin{\cal L}}}
\def\bcM{{\mathbin{\bs{\cal M}}}}
\def\cN{{\mathbin{\cal N}\kern .04em}}
\def\cP{{\mathbin{\cal P}}}
\def\cQ{{\mathbin{\cal Q}}}
\def\cR{{\mathbin{\cal R}}}
\def\cS{{\mathbin{\cal S}}}
\def\T{{{\cal T}\kern .04em}}
\def\cW{{\mathbin{\cal W}}}
\def\cX{{\cal X}}
\def\cY{{\cal Y}}
\def\cZ{{\cal Z}}
\def\oM{{\mathbin{\smash{\,\,\overline{\!\!\mathcal M\!}\,}}}}
\def\cV{{\cal V}}
\def\cW{{\cal W}}
\def\g{{\mathfrak g}}
\def\h{{\mathfrak h}}
\def\m{{\mathfrak m}}
\def\u{{\mathfrak u}}
\def\so{{\mathfrak{so}}}
\def\su{{\mathfrak{su}}}
\def\sp{{\mathfrak{sp}}}
\def\fW{{\mathfrak W}}
\def\fX{{\mathfrak X}}
\def\fY{{\mathfrak Y}}
\def\fZ{{\mathfrak Z}}
\def\bM{{\bs M}}
\def\bN{{\bs N}}
\def\bO{{\bs O}}
\def\bQ{{\bs Q}}
\def\bS{{\bs S}}
\def\bU{{\bs U}}
\def\bV{{\bs V}}
\def\bW{{\bs W}\kern -0.1em}
\def\bX{{\bs X}}
\def\bY{{\bs Y}\kern -0.1em}
\def\bZ{{\bs Z}}
\def\al{\alpha}
\def\be{\beta}
\def\ga{\gamma}
\def\de{\delta}
\def\io{\iota}
\def\ep{\epsilon}
\def\la{\lambda}
\def\ka{\kappa}
\def\th{\theta}
\def\ze{\zeta}
\def\up{\upsilon}
\def\vp{\varphi}
\def\si{\sigma}
\def\om{\omega}
\def\De{\Delta}
\def\La{\Lambda}
\def\Om{\Omega}
\def\Ga{\Gamma}
\def\Si{\Sigma}
\def\Th{\Theta}
\def\Up{\Upsilon}
\def\Chi{{\rm X}}
\def\Tau{{T}}
\def\Nu{{\rm N}}
\def\pd{\partial}
\def\ts{\textstyle}
\def\st{\scriptstyle}
\def\sst{\scriptscriptstyle}
\def\w{\wedge}
\def\sm{\setminus}
\def\lt{\ltimes}
\def\bu{\bullet}
\def\sh{\sharp}
\def\di{\diamond}
\def\he{\heartsuit}
\def\op{\oplus}
\def\od{\odot}
\def\op{\oplus}
\def\ot{\otimes}
\def\bt{\boxtimes}
\def\ov{\overline}
\def\bigop{\bigoplus}
\def\bigot{\bigotimes}
\def\iy{\infty}
\def\es{\emptyset}
\def\ra{\rightarrow}
\def\rra{\rightrightarrows}
\def\Ra{\Rightarrow}
\def\Longra{\Longrightarrow}
\def\ab{\allowbreak}
\def\longra{\longrightarrow}
\def\hookra{\hookrightarrow}
\def\dashra{\dashrightarrow}
\def\lb{\llbracket}
\def\rb{\rrbracket}
\def\ha{{\ts\frac{1}{2}}}
\def\t{\times}
\def\ci{\circ}
\def\ti{\tilde}
\def\d{{\rm d}}
\def\md#1{\vert #1 \vert}
\def\ms#1{\vert #1 \vert^2}
\def\bmd#1{\big\vert #1 \big\vert}
\def\bms#1{\big\vert #1 \big\vert^2}
\def\an#1{\langle #1 \rangle}
\def\ban#1{\bigl\langle #1 \bigr\rangle}

\def\L{\mathcal{L}}
\def\scrA{\mathcal{A}}
\def\scrB{\mathcal{B}}
\def\scrO{\mathcal{O}}
\def\std{\mathrm{std}}
\def\Cl{\mathbin{\mathrm{C}\ell}}
\def\D{\mathbin{\slashed{\operatorname{D}}}}
\let\paragraphS\S
\renewcommand\S{\slashed{\mathrm{S}}}

\def\O{\mathcal{O}}
\def\CP{\mathbb{C}P}
\def\Lambdatop{\bigwedge}
\def\PsiDO{\Psi\mathrm{DO}}

\def\Kappa{\mathrm{K}}

\def\Ell{{\mathcal{E}\ell\ell}}

\def\ITEMorient{\textup{(or)}}
\def\ITEMpfaff{\textup{(pf)}}
\def\ITEMspectral{\textup{(sp)}}
\def\ITEMgeneral{\textup{($\general$)}}

%%%%%%%%%%%%%%%%%%%%%%%%%%%%%%%%%%%%%%%%%%%%%%%%%%%%%%%%%%%%%%%%%%%%%%%%
%%%%%%%%%%%%%%%%%%%%%%%    Text of paper    %%%%%%%%%%%%%%%%%%%%%%%%%%%%
%%%%%%%%%%%%%%%%%%%%%%%%%%%%%%%%%%%%%%%%%%%%%%%%%%%%%%%%%%%%%%%%%%%%%%%%
\title{A categorified excision principle\\
for elliptic symbol families}
\author{Markus Upmeier}
\date{\today}
\maketitle

\begin{abstract}
We develop a categorical index calculus for elliptic symbol families.
The categorified index problems we consider are a secondary
version of the traditional problem of expressing the index class
in $K$-theory in terms of differential-topological data.
They include orientation problems for moduli spaces as well as
similar problems for skew-adjoint and self-adjoint operators.
The main result of this paper is an excision principle which allows
the comparison of categorified index problems on different manifolds.
Excision is a powerful technique for actually solving the orientation problem;
applications appear in the companion papers Joyce--Tanaka--Upmeier~\cite{JTU},
Joyce--Upmeier~\cite{JoUp}, and Cao--Joyce~\cite{CaJo}.
\end{abstract}

%\setcounter{tocdepth}{1}
%\tableofcontents

\section{Introduction}

Index theory assigns the numerical invariant
$\ind P = \dim \Ker P - \dim \Coker P$
to an elliptic operator $P$ on a compact manifold $X$.
The Atiyah--Singer index theorem \cite{AtSi1} solves the
index problem of identifying this invariant in terms
of topological data on $X$. A key technique in its proof
is the following excision principle of Seeley~\cite[A.1]{Seel}.
Let $U$ be contained as an open subset $U^\pm$ in two compact
manifolds $X^\pm$. Let $P^\pm, Q^\pm$ be a pair of elliptic
operators on each of the manifolds $X^\pm$, each pair differing at most on $U$.
%:
Then, deforming the problem through pseudo-differential operators,
the index difference $\ind Q^\pm - \ind P^\pm$ can be
concentrated on $U^\pm$, see for example~\cite[\paragraphS7.1]{DoKr}. In particular,
if we assume $P_+ = P_-$ and $Q_+ = Q_-$ over $U$
it is not hard to believe the excision formula
\e\label{classical-excision}
	\ind Q_-- \ind P_- = \ind Q_+ - \ind P_+.
\e
The excision principle is one of the most powerful techniques of index theory.
Once it is established, one can compare the index problem on any compact
manifold to that on a sphere, where it is solved using Bott periodicity.\medskip

In proving \eqref{classical-excision} and extending to families the key is to
establish a formal calculus for the index map in terms of topological $K$\nobreakdash-theory.
In this calculus $\ind P$ depends only on the principal symbol, is functorial, and
can be extended to open manifolds. These properties lead to \eqref{classical-excision}
and, combined with the multiplicative property, to the families index theorem.\medskip

In his study of orientations in Yang--Mills theory, Donaldson introduced in
\cite[(3.10)]{Dona2} a formally similar excision isomorphism for the determinant
\[
	\det D_-
	\longrightarrow
	\det D_+,
\]
where $D_\pm$ are first order differential operators on $X_\pm$, skew-adjoint outside $U_\pm$.
It is based on an adiabatic construction of solutions to the differential equation $D_+f_+=0$,
given a solution $D_- f_- = 0$. An alternative proof closer to our discussion here in terms of
pseudo-differential operators is explained in Donaldson--Kronheimer \cite[\paragraphS7.1]{DoKr}.\medskip

In this paper, we shall expand these ideas and develop a new `categorical index calculus'
in Theorem~\ref{thm:symbol-calculus}.
In a categorified index problem one assigns, rather than numbers, $G$\nobreakdash-torsors
to Fredholm operators ($G = \Z_2, \Z$). We shall consider the following three basic cases
$\lambda \in \{\orient, \pfaff, \spectral\}$:
\begin{enumerate}
\item[{\ITEMorient}] The \emph{orientation problem} for real Fredholm operators to which we assign
the $\Z_2$\nobreakdash-torsor of orientations $\orient(P)$ on the
vector space ${\Ker P \oplus (\Coker P)^*}$.
\item[{\ITEMpfaff}] The \emph{skew orientation problem} for real skew-adjoint Fredholm operators
to which we assign the $\Z_2$\nobreakdash-torsor of orientations $\pfaff(P)$ on ${\Ker P}$.
\item[{\ITEMspectral}]  The \emph{spectral orientation problem} for self-adjoint Fredholm operators
to which we assign their spectral $\Z$\nobreakdash-torsor $\spectral(P)$, see Definition~\textup{\ref{dfn:spectral-torsor}}.
\end{enumerate}
Thus $\general(P)$ plays the role of $\ind(P)$ and the category of $G$-torsors plays the
role of the topological $K$-theory group. A key point is that all three orientation problems
are deformation invariant. In particular, restricted to pseudo-differential operators they only
depend on principal symbols, see Section~\ref{sssec:covering-symbol}, which is the starting
point for the categorical index calculus in Section~\ref{ssec:proof-calc}.\medskip

Rather than to connected components $\pi_0$ these problems now correspond to
$\pi_1 \Fred_\R = \Z_2$,
$\pi_1 \Fred_\R^\mathrm{skew} = \Z_2$,
and $\pi_1 \Fred_\C^\mathrm{sa}=\Z$.
In the first two cases, the universal covers are known as the real
determinant line bundle and the real Pfaffian line bundle, see for example
Freed~\cite{Fre}, turned into double covers. The universal cover
of $\Fred_\C^\mathrm{sa}$
is less familiar and in Section~\ref{ssec:spectral-cover} we construct the
spectral cover for self-adjoint families, a principal $\Z$\nobreakdash-bundle. This
construction is new and in Theorem~\ref{thm:sa-generalFred} we establish also
its connection to the transgression of the complex determinant line bundle.
The last case {\ITEMspectral} has potential applications to orientations graded
over $\Z_n$ for $n\in \N$.\medskip

The categorical index calculus is then applied to prove Theorem~\ref{main-theorem}.
In all basic cases $\lambda \in \{\orient, \pfaff, \spectral\}$
we establish a \emph{canonical excision isomorphism}
\[
	\Ex\colon \general (P_+)^* \otimes \general (Q_+)
	\longrightarrow
	\general (P_-)^* \otimes \general (Q_-)
\]
together with its functoriality properties and its compatibilities with other constructions.
These properties will be crucial in Joyce--Upmeier~\cite{JoUp} to
solve the orientation problem for twisted Dirac operators $P=\D_{\Ad E}$
on a $7$\nobreakdash-dimensional compact spin manifold in terms of a flag structure on
$X$, with applications to the study of moduli of $G_2$\nobreakdash-instantons \cite{DoSe}.\medskip

Categorified index problems are
secondary index problems, meaning they only make sense when the corresponding
traditional index vanishes. For {\ITEMorient} this is the orientability of the real
determinant line bundle, which amounts to the vanishing of the first Stiefel--Whitney
class of the index class $\ind P \in KO^0(Y)$ from \cite{AtSi5}.
For {\ITEMpfaff} and {\ITEMspectral} we have by \cite{AtSiSkew} and
by \cite[\paragraphS3]{AtPaSi3} an index in $KO^{-1}(Y)$
and $K^{-1}(Y)$ whose images in $H^1(Y;\Z_2)$ and $H^1(Y;\Z)$ must vanish.
Orientability questions can thus be treated by classical index theory, but for
questions of picking actual orientations this is no longer sufficient.
To do this, traditional equalities must be replaced by \emph{canonical isomorphisms}.

\subsubsection*{Outline of the paper}

In Section~\ref{ssec:elliptic-symbol-families} we set up the necessary terminology for
elliptic symbols to formulate the categorical
index calculus. This is done in Section~\ref{ssec:main-results} as Theorem~\ref{thm:symbol-calculus},
where we state also our other main result, Theorem~\ref{main-theorem}, on excision.
A simplified
version of the excision principle for gauge theory is stated as Theorem~\ref{thm:gauge-special}.
Assuming the categorical index calculus, we prove
Theorem~\ref{main-theorem} in Section~\ref{ssec:proofMAIN}.

The rest of the paper establishes the categorical index calculus
in the three basic cases.
Our
orientation conventions are fixed in Section~\ref{ssec:sign-conventions}.
Sections~\ref{ssec:Determinant} and \ref{ssec:Pfaffian} review known results for
the determinant and Pfaffian line bundles.
The spectral cover for self-adjoint Fredholm operators is constructed in Section~\ref{ssec:spectral-cover}.
As an aside, the relationship with the transgression of the complex determinant line
clarified in Theorem~\ref{thm:sa-generalFred}. After these preparations, we then
prove Theorem~\ref{thm:symbol-calculus}, the categorical symbol calculus, in Section~\ref{ssec:proof-calc}.

For convenience, we have collected in Appendix~\ref{sec-symbols-and-operators} some elementary
background  on compactly supported pseudo-differential operators.\\

This is the first of a series of papers on orientations in gauge theory.
The second paper Joyce--Tanaka--Upmeier~\cite{JTU} establishes
the general theory and gives examples of solutions to
orientation problems in dimensions up to $6$, some of which are new.
The third paper Joyce--Upmeier~\cite{JoUp} solves the orientation problem for Dirac operators
in dimension $7$.
In the fourth paper, Cao--Joyce \cite{CaJo} will prove orientability of the
moduli space of $\Spin(7)$\nobreakdash-instantons in dimension~$8$ using
our excision principle. Finally, \cite{CaGrJo} will deal with the connection
to algebraic geometry.\medskip

\noindent{\it Acknowledgements.}
The author was funded by DFG grant UP~85/3-1, 
by grant UP~85/2-1 of the DFG priority program SPP~2026 `Geometry at Infinity,' and
by the `Centre for Quantum Geometry of Moduli Spaces' of the DNRF.
The author would like to thank Dominic Joyce for numerous discussions.
The author would also like to thank Yalong Cao, Simon Donaldson,
Sebastian Goette, Jacob Gross, Yuuji Tanaka, and Thomas Walpuski
for helpful conversations.

\section{Categorical index calculus}

\subsection{Elliptic symbol families}\label{ssec:elliptic-symbol-families}

\subsubsection{The category of elliptic symbol families}

For the precise meaning of (continuous) $Y$\nobreakdash-families of smooth objects parameterized by a
topological space $Y$ we refer to Appendix~\ref{families}.

\begin{dfn}[{Atiyah--Singer~{\cite[p.~491]{AtSi1}}}]
Let $Y$ be a space, $X$ a manifold, not necessarily compact,
and $\pi\colon \left(T^*X \setminus 0_X\right) \times Y\to X\times Y$ the projection.
A \emph{$Y$\nobreakdash-family $\{p_y\}_{y \in Y}$ of elliptic symbols over $X$
of order $m\in \R$} consists of Hermitian vector bundles 
$E^0, E^1 \to X \times Y$ and a $Y$\nobreakdash-family of isomorphisms
\[
	p\colon \pi^*E^0 \longrightarrow \pi^*E^1
\]
satisying
$p_{\lambda\cdot \xi,y} = \lambda^m \cdot p_{\xi,y}$
for all $0\neq \xi \in T^*X, y\in Y, \lambda > 0$
on the fibers.
The $Y$\nobreakdash-families of $m$\nobreakdash-th order elliptic symbols form the set $\Ell_Y^m(X;E^0,E^1)$.
We consider the following \emph{types} of elliptic symbols:
\begin{enumerate}
	\item[\ITEMorient]
	$p$ is \emph{real} if $E^0, E^1$ have orthogonal real structures and
	$\overline{p_{\xi,y}(e)} = p_{-\xi,y}(\overline{e})$.
	\item[\ITEMpfaff]
	$p$ is \emph{real skew-adjoint} if $p$ is real and $(p_{\xi,y})^\dagger = - p_{\xi,y}$.
	\item[\ITEMspectral]
	$p$ is \emph{self-adjoint} if $(p_{\xi,y})^\dagger = p_{\xi,y}$.
\end{enumerate}
\end{dfn}

Direct sums and adjoints of elliptic symbols are formed pointwise.

\begin{dfn}
%\hangindent\leftmargini
%\textup{(i)}\hskip\labelsep
Let $p_\pm \in \Ell_Y^m(X_\pm;E^0_\pm, E^1_\pm)$ be families
of elliptic symbols of the same order $m$ on manifolds $X_\pm$.
Let $\phi \colon X_- \to X_+$ be an open embedding.
An \emph{identification $\Phi\colon p_- \to p_+$ of symbols over $\phi$}
are two $Y$\nobreakdash-families of unitary isomorphisms
$\Phi^\bu\colon E^\bu_- \to \phi^*E^\bu_+$, $\bu=0,1$, satisfying
\[
\forall 0\neq \xi\in T^*X_-, y \in Y:
	(p_+)_{d\phi(\xi),y}\circ \Phi^0 = \Phi^1\circ (p_-)_{\xi,y}.
\]
We consider the following \emph{types} of identifications:
$\Phi$ is
{\ITEMorient} \emph{real} if $E^\bu_\pm$ have orthogonal real structures with which $\Phi^0, \Phi^1$ commute,
{\ITEMpfaff} \emph{real skew-adjoint} if it is real and $\Phi^0 = -\Phi^1$, and
{\ITEMspectral} \emph{self-adjoint} if $\Phi^0 = \Phi^1$.
%\end{enumerate}
\end{dfn}

There is then a monoidal category $\Ell_Y^m(X)$ whose
objects are $m$\nobreakdash-th order $Y$\nobreakdash-families of elliptic symbols.
The morphisms are identifications over ${\phi=\id_X}$. An embedding
${\phi\colon X_- \to X_+}$ induces a functor
${\phi^*\colon \Ell_Y^m(X_+) \to \Ell_Y^m(X_-)}$, so $\Ell_Y^m(\cdot)$ organizes into
a presheaf of groupoids on the site of manifolds. The decategorification
$\pi_0\Ell_Y^m(X)$ is $K^0_\mathrm{cpt}(Y\times T^*X)$, see for
example \cite[p.~100]{LaMi}.

\begin{prop}\label{prop:red-of-order}
We have a reduction of order functor:
\begin{enumerate}
\item[\textup{(i)}]
If $p \in \Ell_Y^m(X;E^0, E^1)$, then $p(p^\dagger p)^{-1/2} \in \Ell_Y^0(X; E^0, E^1)$.
\item[\textup{(ii)}]
If $\Phi \colon p_- \to p_+$ in $\Ell_Y^m(X)$, then $\Phi\colon p_-(p_-^\dagger p_-)^{-1/2} \to p_+(p_+^\dagger p_+)^{-1/2}$ in $\Ell_Y^0(X)$.
\end{enumerate}
\end{prop}

Finally, we introduce the following
terminology for dealing with zeroth-order symbols on
open manifolds.

\begin{dfn}
Let $L \subset X$ be a compact set and $p \in \Ell_Y^0(X;E^0,E^1)$.
Then $p$ is \emph{compactly supported} in $L$ if
there exists a (unique) $Y$\nobreakdash-family of
bundle isomorphism ${\widetilde{p}\colon E^0|_{(X \setminus L) \times Y} \to E^1|_{(X \setminus L) \times Y}}$ satisfying the following conditions:
\begin{gather}
\forall (x,y) \notin L, \xi \in T^*_x X : p_{\xi, y} = \widetilde{p}_{x,y}.\nonumber\\
\forall y \in Y \mkern2mu\exists c, C > 0 \mkern4mu\forall x \notin L_y: c \leq \|\widetilde{p}_{x,y}\| \leq C.\label{dfn.symb-compactly-supp-bounds}\\
\begin{gathered}\nonumber
\text{$\forall \varepsilon > 0, y_0 \in Y$ there exists a neighborhood $V$ of $y_0$ such that:}\\
(x,y) \in (X \times V) \setminus L
\implies
\|\widetilde{p}_{x,y} - \widetilde{p}_{x,y_0}\| \leq \varepsilon.
\end{gathered}
\end{gather}
For $Z$ compact Hausdorff and $p \in \Ell_{Y\times Z}(X;E^0,E^1)$
we more generally allow $L\subset X\times Z$ to vary
along the compact $Z$ parameter.
\end{dfn}

\subsubsection{The category of relative pairs}

The terminology of this section is not needed for the categorical index calculus,
but is used in the formulation of the excision theorem.

An identification corresponds
to the data that identifies pseudo-differential operators
$P_\pm$ on $X_\pm$. As such, they induce isomorphisms
${\Phi^*\colon \general(p_+) \to \general(p_-)}$, see Theorem~\ref{thm:symbol-calculus}(i).
Excision extends this \emph{global} functoriality to the case where we have
a pair $(P_\pm, Q_\pm)$ of operators on each of the spaces $X_\pm$ and where the diffeomorphism
$\phi$ only needs to be defined \emph{locally} where the operators differ, see Figure~\ref{fig1}.
The idea is that the index-theoretic information contained in $P_-$ and $Q_-$
can be compressed close to where the operators differ. Similarly for $P_+$ and $Q_+$. The compressed solutions to the (pseudo-)differential equations can then
be mapped back and forth using a only locally defined diffeomorphism.

The data needed to perform the compression in a canonical way promotes
the pair $(p,q)$ of principal symbols to a relative pair $(p,\Xi,q)$:

\begin{dfn}\label{dfn:relative-pair}
Let $p \in \Ell_Y^m(X; E^0, E^1)$, $q \in \Ell_Y^m(X; F^0, F^1)$.
An identification
$\Xi\colon p|_{X\setminus L} \to q|_{X\setminus L}$ over the identity
defined outside a closed subset ${L\subset X}$ is called a
\emph{relative $Y$\nobreakdash-pair} $(p,\Xi,q)$ with \emph{support} $L$.
A relative pair has \emph{type} {\ITEMgeneral} if all of $p, q, \Xi$
have type {\ITEMgeneral}.
\end{dfn}

The motivation for this terminology is that
the identification $\Xi$
promotes the isomorphism class
$[p]-[q] \in K^0_\mathrm{cpt}(Y\times T^*X)$ 
to a relative cohomology class in
$K^0_\mathrm{cpt}\left(Y\times T^*X, Y\times T^*(X\setminus L)\right)$.
We shall see that $\Xi$ allows us to `compress' the orientation cover
$\general(p)^*\otimes\general(q)$
into a neighborhood of the support of $\Xi$.
To make this idea precise, we introduce the following notion.

\begin{dfn}\label{dfn:iso-of-rel-pairs}
Let $(p_\pm, \Xi_\pm, q_\pm)$ be relative $Y$\nobreakdash-pairs
on manifolds $X_\pm$ with compact supports $L_\pm$
and $\phi \colon U_- \to U_+$ a diffeomorphism of
open sets ${U_\pm \subset X_\pm}$ containing $L_\pm$.
An \emph{isomorphism of relative $Y$\nobreakdash-pairs} over the base map $\phi$, written
\e\label{eqn:concentrated-symb-iso}
\xymatrix@C=15ex{
(p_-,\Xi_-,q_-)|_{U_-} 
\ar[r]|{(\phi,\Pi,\Kappa)}
&(p_+,\Xi_+,q_+)|_{U_+},
}
\e
consists of two identifications
${\Pi\colon p_-|_{U_-} \to p_+|_{U_+}}$, ${\Kappa\colon q_-|_{U_-} \to q_+|_{U_+}}$ over $\phi$
satisfying $\Xi_+\circ \Pi|_{U_-\setminus C} = \Kappa \circ \Xi_-|_{U_-\setminus C}$ outside a compact subset $C \subset U_-$ with $L_- \subset C$ and $L_+ \subset \phi(C)$.
We then say the \emph{support} of \eqref{eqn:concentrated-symb-iso} is contained in $C$.
By enlarging the supports of the relative pairs we can always assume $L_- = C, L_+ = \phi(C)$.
% i.e. K is the support
% i.e. on $U_-\setminus K$
%The following is stricter in that it assumes that $\phi$ maps $L\pm$ onto each other
%$\Xi_+|_{U_+ \setminus L_+}\circ \Pi|_{U_-\setminus L_-} = \Kappa|_{U_-\setminus L_-}\circ \Xi_-|_{U_-\setminus L_-}$.
An isomorphism \eqref{eqn:concentrated-symb-iso} is
of \emph{type} {\ITEMgeneral} if $(p_\pm, \Xi_\pm, q_\pm)$ and $\Pi, \Kappa$ are all of
type {\ITEMgeneral}.
\end{dfn}

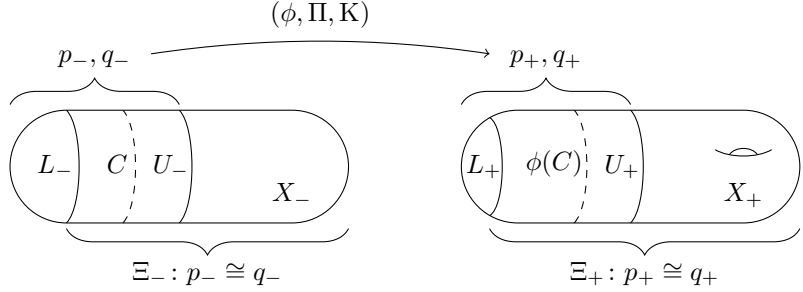
\begin{figure}[h]
\centering
\begin{tikzpicture}[scale=0.75]

%BEGIN LEFT
\begin{scope}
\begin{scope}
% The pill
\draw(-2,1) -- (2,1) arc (90:-90:1) -- (-2,-1) arc (270:90:1);
\draw (2,-0.5) node {$X_-$};

% Set U
\draw (0,-1) .. controls (0.3,-0.7) and (0.3,0.7) .. (0,1); \draw (-0.15,0) node{$U_-$};

% Set C
\begin{scope}[xshift=-1cm]
	\draw[dashed] (0,-1) .. controls (0.3,-0.7) and (0.3,0.7) .. (0,1);
\end{scope}
\draw (-1.1,0.05) node{$C$};

% Set L
\begin{scope}[xshift=-2cm]
	\draw (0,-1) .. controls (0.3,-0.7) and (0.3,0.7) .. (0,1); \draw (-0.2,0) node{$L_-$};
\end{scope}

% Brace below
\draw[decorate, decoration={brace,amplitude=10pt}] (3,-1.1) -- (-2,-1.1) node[midway,yshift=-0.6cm]{$\Xi_-\colon p_- \cong q_-$};

% Brace above
\draw[decorate, decoration={brace,amplitude=10pt}] (-3,1.1) -- (0,1.1) node[midway,yshift=0.6cm]{$p_-, q_-$};
\end{scope}
\end{scope}
%END LEFT

% The arrow between left and right
\draw[->](-0.5,2) .. controls (1.5,2.3) and (3.5,2.3) .. (5.5,2);
\draw (2.5,2.7) node{$(\phi,\Pi,\Kappa)$};

%BEGIN RIGHT
\begin{scope}[xshift=8cm]
\begin{scope}
% The pill
\draw(-2,1) -- (2,1) arc (90:-90:1) -- (-2,-1) arc (270:90:1);
\draw (2,-0.5) node {$X_+$};

% The hole

\draw (1.5,0.3) .. controls (1.75,0.15) and (2.25,0.15) .. (2.5,0.3);

\begin{scope}[yshift=-0.09cm]
\draw (1.75,0.3) .. controls (1.85,0.45) and (2.15,0.45) .. (2.25,0.3);
\end{scope}

% Set U
\draw (0,-1) .. controls (0.3,-0.7) and (0.3,0.7) .. (0,1); \draw (-0.15,0) node{$U_+$};

% Set C
\begin{scope}[xshift=-1cm]
	\draw[dashed] (0,-1) .. controls (0.3,-0.7) and (0.3,0.7) .. (0,1); \draw (-0.35,0.05) node{$\phi(C)$};
\end{scope}

% Set L
\begin{scope}[xshift=-2.5cm]
	\draw (0,-0.87) .. controls (0.3,-0.7) and (0.3,0.7) .. (0,0.87); \draw (-0.1,0) node{$L_+$};
\end{scope}

% Brace below
\draw[decorate, decoration={brace,amplitude=10pt}] (3,-1.1) -- (-2.5,-1.1) node[midway,yshift=-0.6cm]{$\Xi_+\colon p_+ \cong q_+$};

% Brace above
\draw[decorate, decoration={brace,amplitude=10pt}] (-3,1.1) -- (0,1.1) node[midway,yshift=0.6cm]{$p_+, q_+$};
\end{scope}
\end{scope}
%END RIGHT

\end{tikzpicture}
\caption{An isomorphism of relative pairs.}
\label{fig1}
\end{figure}

To this setup, depicted in Figure~\ref{fig1}, we shall attach an excision isomorphism.
We will also prove its independence under the following deformations.

\begin{dfn}
Let $Y$ be a topological space and $Z$ compact Hausdorff.
There is a version of Definition~\ref{dfn:relative-pair}
in which $L\subset X\times Z$ may change along the
compact parameter $z \in Z$.
Similarly, in Definition~\ref{dfn:iso-of-rel-pairs}
we can take $U_\pm \subset X_\pm \times Z$ open and $\phi\colon U_-\to U_+$ to
be a $Z$\nobreakdash-family of diffeomorphisms. In this case
we speak of a \emph{$Z$\nobreakdash-deformation} of relative $Y$\nobreakdash-pairs
\e\label{eqn:Zdeformation}
\xymatrix@1@C=15ex{\ar[r]|{(\phi_z,\Pi_z,\Kappa_z)} (p_{-,z},\Xi_{-,z}, q_{-,z}) &(p_{+,z},\Xi_{+,z}, q_{+,z})},\quad z\in Z.
\e
\end{dfn}

%Of course, for relative pairs we again have a category of relative symbols and a reduction of order functor.

\subsection{Statement of results}\label{ssec:main-results}

\subsubsection{Index calculus}

We now present the calculus that is a categorical version
of the $K$\nobreakdash-theory calculus for the index map.
Proofs are given in Section~\ref{ssec:proof-calc}.

\begin{thm}\label{thm:symbol-calculus}
Let $X$ be a manifold, $Y$ a topological space, and $\general \in \{\orient, \pfaff, \spectral\}$.
To every $Y$\nobreakdash-family of elliptic symbols $p \in \Ell_Y^m(X)$ of type {\ITEMgeneral},
zeroth-order compactly supported if $X$ is not compact, there is associated a $\Z_2$\nobreakdash-graded
$G$\nobreakdash-principal bundle $\general(p) \to Y$ with
\e\label{eqn:degree}
	\deg \general(p)
	\overset{2}{\equiv}
	\begin{cases}
	\ind (p),\\
	\dim_\R \Ker(p),\\
	0,
	\end{cases}
	\qquad
	G=
	\begin{cases}
	\Z_2 & \ITEMorient,\\
	\Z_2 & \ITEMpfaff,\\
	\Z & \ITEMspectral,
	\end{cases}
\e
and the following formal properties:
\begin{enumerate}
\item[\textup{(i)}]
An identification $\Phi\colon p_- \to p_+$ of type {\ITEMgeneral} covering a diffeomorphism $\phi\colon X_- \to X_+$
induces a \emph{functoriality isomorphism}
\e\label{eqn:maps-induced-by-identification}
	\Phi^*
	\colon
	\general(p_+)
	\longrightarrow
	\general(p_-)
\e
satisfying $(\Phi\circ \Psi)^* = \Psi^*\circ\Phi^*$ and $\id^* = \id$.
\item[\textup{(ii)}]
For $p, q \in \Ell_Y^m(X)$ of type {\ITEMgeneral} we have a \emph{direct sum isomorphism}
\e\label{eqn:symbols-direct-sums}
	\general(p\oplus q)
	\longrightarrow
	\general(p)\otimes\general(q),
\e
which is associative, graded commutative, and natural for \eqref{eqn:maps-induced-by-identification}.
\item[\textup{(iii)}]
For $p \in \Ell_Y^m(X)$ of type {\ITEMgeneral} we have an \emph{adjointness isomorphism}
\e\label{eqn:symbols-adjoints}
	\general(-p^\dagger)
	\longrightarrow
	\general(p)^*,
\e
natural for \eqref{eqn:maps-induced-by-identification} and compatible with \eqref{eqn:symbols-direct-sums}.
\item[\textup{(iv)}]
Let $p \in \Ell_Y^m(X)$ be of type {\ITEMgeneral}.
An open embedding $i\colon U \hookrightarrow X$ induces a \emph{pushforward isomorphism}
\e\label{eqn:pushforwards}
	i_!\colon
	\general(i^*p)
	\longrightarrow
	\general(p),
\e
natural for \eqref{eqn:maps-induced-by-identification}, compatible with \eqref{eqn:symbols-direct-sums} and \eqref{eqn:symbols-adjoints}, and satisfying
\[
	i_!j_! = (ij)_!,
	\qquad
	\id_! = \id.
\]
\item[\textup{(v)}]
For $p \in \Ell_Y^m(X)$ of type {\ITEMgeneral} we have
$p(p^\dagger p)^{-1/2} \in \Ell_Y^0(X)$ of type {\ITEMgeneral}.
There is a \emph{reduction of order isomorphism}
\e\label{eqn:symbols-reduction-of-order}
	\general(p)
	\longrightarrow
	\general(p(p^\dagger p)^{-1/2}),
\e
compatible with \eqref{eqn:symbols-direct-sums}--\eqref{eqn:pushforwards}.
An identification $\Phi \colon p_- \to p_+$ of type {\ITEMgeneral}
is also an identification $p_-(p_-^\dagger p_-)^{-1/2}
\to p_+(p_+^\dagger p_+)^{-1/2}$ of the same type, and then
\eqref{eqn:symbols-reduction-of-order} is natural for \eqref{eqn:maps-induced-by-identification}.
\end{enumerate}
\end{thm}

Of course, all isomorphisms are understood as grading-preserving
isomorphisms of $G$\nobreakdash-principal bundles. The tensor product in \eqref{eqn:symbols-direct-sums}
is the structure group reduction of the direct product $G\times G$\nobreakdash-principal bundle along the group operation map $G \times G \to G$,
while the dual in \eqref{eqn:symbols-adjoints} is obtained by structure group reduction along
inversion $G\to G$.

\begin{cor}\label{cor:deformations}
Let $Z$ be compact Hausdorff.
For $p \in \Ell^m_{Y\times Z}(X)$ of type {\ITEMgeneral},
compactly supported if $X$ is not compact, the disjoint union
$\bigsqcup_{z\in Z} \general(p_z)$ has a unique natural topology
of $G$\nobreakdash-principal bundle over $Y\times Z$ extending that of Theorem~\textup{\ref{thm:symbol-calculus}}.
In particular, for $Z=[0,1]$ we get a fiber transport isomorphism
\e\label{eqn:symbol-deformations}
	\general\{p_z\}_{z \in [0,1]}\colon \general(p_0) \longrightarrow \general(p_1),
\e
compatible with \eqref{eqn:maps-induced-by-identification}--\eqref{eqn:symbols-reduction-of-order}.
\end{cor}

\begin{proof}
If $X$ is compact, this follows by applying Theorem~\ref{thm:symbol-calculus} to $Y\times Z$.
Similarly for $m=0$ and $p$ compactly supported in $L=K\times Z$ with $K\subset X$ compact.
The general case can be reduced to this, since every $z_0 \in Z$ has a neighborhood $Z_0 \subset Z$ such that
$\bigcup_{z\in Z_0} L_z \subset \bigcap_{z\in Z_0} U_z$.
The restriction of $p$ to a $Y\times Z_0$\nobreakdash-family is then of the previous type.
\end{proof}

\subsubsection{Excision}

\begin{thm}\label{main-theorem}
Let $X_\pm$ be compact manifolds and $Y$ a topological space.
Every isomorphism $(\phi,\Pi,\Kappa)$ of relative pairs  of type {\ITEMgeneral}
as in \eqref{eqn:concentrated-symb-iso}
induces an \emph{excision isomorphism} of $\Z_2$\nobreakdash-graded $G$\nobreakdash-principal bundles over $Y$,
\e\label{eqn:excision-iso}
%	\Ex(\phi, \Pi, \Kappa, \Xi_-, \Xi_+)\colon
\begin{aligned}
\Ex(\phi,\Pi,\Kappa)\colon
	\general (p_-)^* \otimes \general (q_-)
	\longrightarrow
%	\Ex\left(\scriptsize
%	\xymatrix@C=13ex@1{(p_-,\Xi_-,q_-)|_{U_-}\ar[r]|{(\phi,\Pi,\Kappa)} & (p_+,\Xi_+,q_+)|_{U_+}}
%	(p_-,\Xi_-,q_-) \xrightarrow{(\phi,\Pi,\Kappa)} (p_+,\Xi_+,q_+)	
%	\right)
	\general (p_+)^* \otimes \general(q_+),
\end{aligned}
\e
uniquely determined by the following properties:
\begin{enumerate}
\item[\textup{(i)}] \textup{(Empty support.)}~If $(\phi,\Pi,\Kappa)$ has empty support, then 
$\general (p_\pm)^* \otimes \general (q_\pm)$ are both
identified with the trivial cover. In this identification the excision isomorphism
$\Ex(\phi,\Pi,\Kappa)$ becomes the identity map.
\item[\textup{(ii)}] \textup{(Restriction.)}~Let $V_\pm \supset U_\pm$ be open supersets.
Assume that $(\phi,\Pi,\Kappa)$
has an extension $(\hat\phi, \hat\Xi, \hat\Kappa)$
to an isomorphism of relative pairs over a base diffeomorphism
$\hat\phi \colon V_- \to V_+$ of the same type {\ITEMgeneral}. Then
\[
	\Ex(\phi,\Pi, \Kappa) = \Ex(\hat\phi,\hat\Pi, \hat\Kappa).
\]
\item[\textup{(iii)}] \textup{(Sums.)}~In addition to $(\phi,\Pi,\Kappa)$ let
$\xymatrix@C=11ex@1{\ar[r]|{(\phi,\widetilde\Pi,\widetilde\Kappa)}(\widetilde{p}_-,\widetilde\Xi_-,\widetilde{q}_-) & (\widetilde{p}_+,\widetilde\Xi_+,\widetilde{q}_+)}$
be another isomorphism of relative pairs over the same $\phi$ and type {\ITEMgeneral}.
For the sum of symbols and identifications we have a commutative diagram
\[\xymatrix{
\general (p_-)^*\otimes \general (q_-) \otimes \general (\widetilde{p}_-)^* \otimes \general (\widetilde{q}_-)
\ar[d]|{\Ex(\phi,\Pi,\Kappa) \otimes \Ex(\phi,\widetilde\Pi,\widetilde\Kappa)}
\ar[r]_-{\eqref{eqn:direct-sums-horiz-maps}}
&
\general (p_- \oplus \widetilde{p}_-)^* \otimes \general(q_- \oplus \widetilde{q}_-)
\ar[d]|{\Ex(\phi,\Pi\oplus\widetilde\Pi, \Kappa\oplus\widetilde\Kappa)}
\\
\general (p_+)^*\otimes \general (q_+) \otimes \general (\widetilde{p}_+)^* \otimes \general (\widetilde{q}_+)\ar[r]^-{\eqref{eqn:direct-sums-horiz-maps}}
&
\general (p_+ \oplus \widetilde{p}_+)^* \otimes \general(q_+ \oplus \widetilde{q}_+).\\
}\]
\item[\textup{(iv)}] \textup{(Functoriality.)}~Let $X_-, X_0, X_+$ be compact manifolds. Given composable isomorphisms of relative pairs of the same type {\ITEMgeneral}
\[
\xymatrix@C=18ex{
(p_-,\Xi_-, q_-) \ar[r]|-{(\phi_{-0}, \Pi_{-0}, \Kappa_{-0})}& (p_0, \Xi_0, q_0)
\ar[r]|-{(\phi_{0+}, \Pi_{0+}, \Kappa_{0+})}& (p_+, \Xi_+, q_+)
}
\]
we have for the horizontal composition \textup($\phi_{-+} \coloneqq \phi_{0+}\circ\phi_{-0}$ and so forth\textup)
\[
	\Ex(\phi_{0+}, \Pi_{0+}, \Kappa_{0+})
	\circ
	\Ex(\phi_{-0}, \Pi_{-0}, \Kappa_{-0})
	=
	\Ex(\phi_{-+}, \Pi_{-+}, \Kappa_{-+}).
\]
\item[\textup{(v)}] \textup{(Global excision.)}~If the isomorphism $(\phi,\Pi,\Kappa)$
of relative pairs is defined
over a global diffeomorphism ${\phi\colon X_- \to X_+}$, $U_\pm = X_\pm$,
then excision coincides with the global functoriality
${\Ex(\phi,\Pi,\Kappa)=
(\Pi^*)^* \otimes (\Kappa^*)^{-1}}$
of \eqref{eqn:maps-induced-by-identification}.
\item[\textup{(vi)}] \textup{(Deformations.)}~Consider
a $Z$\nobreakdash-deformation \eqref{eqn:Zdeformation}. Then \eqref{eqn:excision-iso} 
is continuous in the topology of Corollary~\textup{\ref{cor:deformations}}.
In particular, when $Z=[0,1]$ and both relative pairs $(p_{\pm,z},\Xi_{\pm,z}, q_{\pm,z})$
are constant in $z$ we get a homotopy
\[
	\Ex(\phi_z,\Pi_z,\Kappa_z)\colon
	\general (p_-)^* \otimes \general (q_-)
\longrightarrow	\general (p_+)^* \otimes \general(q_+),\qquad z\in [0,1].
\]
\item[\textup{(vii)}] \textup{(Reduction of order.)}~Regarding
$(\phi,\Pi,\Kappa)$ also as an isomorphism of the zeroth-order relative symbols,
we have a commutative diagram
\[\xymatrix{
\general(p_-)^*\otimes \general(q_-)\ar[r]_-{\eqref{eqn:symbols-reduction-of-order}}
\ar[d]_{\Ex(\phi,\Pi,\Kappa)}
&
\general\left(p_-(p_-^\dagger p_-)^{-1/2}\right)^*
\otimes
\general\left(q_-(q_-^\dagger q_-)^{-1/2}\right)
\ar[d]^{\Ex(\phi,\Pi,\Kappa)}
\\
\general(p_+)^*\otimes \general(q_+)\ar[r]^-{\eqref{eqn:symbols-reduction-of-order}}
&
\general\left(p_+(p_+^\dagger p_+)^{-1/2}\right)^*
\otimes
\general\left(q_+(q_+^\dagger q_+)^{-1/2}\right).
}\]
\end{enumerate}
If the involved symbols are zeroth-order and compactly supported, everything holds
equally for non-compact manifolds. 
%For ${\general=\orient}$ we have another result for the vertical composition:
%\begin{enumerate}
%\item[\textup{(vii)}] \textup{(Composition.)}~For a vertical composition \eqref{eqn:vertical-composition} in which
%all of the symbols and isomorphisms are real we have a commutative diagram
%\[\hskip-3ex\xymatrix@C=15ex{
%\orient(p_-^1)^*\otimes \orient(q_-^1)\otimes \orient(p_-^0)^* \otimes \orient(q_-^0)\ar[r]_-{\eqref{eqn:symbol-composition}^{*,-1}\otimes\eqref{eqn:symbol-composition}\circ \sigma }
%\ar[d]|{\Ex(\phi,\Pi^{1,2},\Kappa^{1,2})\otimes \Ex(\phi,\Pi^{0,1},\Kappa^{0,1})}
%&
%\orient(p_-^1p_-^0)^*\otimes\orient(q_-^1q_-^0)\ar[d]|{\Ex(\phi,\Pi^{0,2},\Kappa^{0,2})}\\
%\orient(p_+^1)^*\otimes \orient(q_+^1)\otimes \orient(p_+^0)^* \otimes \orient(q_+^0)
%\ar[r]^-{\eqref{eqn:symbol-composition}^{*,-1}\otimes\eqref{eqn:symbol-composition}\circ \sigma }
%&
%\orient(p_+^1p_+^0)^*\otimes\orient(q_+^1q_+^0),
%}\]
%where $\sigma$ swaps the second factor into last place, see \eqref{L1L2dual}.
%\end{enumerate}
\end{thm}

\begin{rem}
Our sign convention in Section~\textup{\ref{ssec:sign-conventions}}
fixes the map in (iii) as
\e\label{eqn:direct-sums-horiz-maps}
\hskip-2ex\begin{aligned}
&\general (p)^*\otimes \general (q) \otimes \general (\widetilde{p})^* \otimes \general (\widetilde{q})
\xrightarrow{1\otimes \eqref{eqn:tau}\otimes 1}\general (p)^*\otimes \general (\widetilde{p})^* \otimes\general (q) \otimes \general (\widetilde{q})\\
&\xrightarrow{\eqref{L1L2dual}\otimes 1}(\general \widetilde{p}\otimes \general p)^* \otimes\general (q) \otimes \general (\widetilde{q})
\xrightarrow{\eqref{eqn:tau}^*\otimes 1\otimes 1}
(\general p\otimes \general \widetilde{p})^* \otimes\general (q) \otimes \general (\widetilde{q})\\
&\xrightarrow{\eqref{eqn:symbols-direct-sums}^{*,-1}\otimes\eqref{eqn:symbols-direct-sums}}\general(p\oplus \widetilde{p})^* \otimes \general(q\oplus\widetilde{q}).
\end{aligned}
\e
Thus, for $\general=\orient$ the diagram in (iii) includes
$
	(-1)^{(\ind p_\pm - \ind q_\pm)\cdot (\ind\widetilde{p}_- + \ind\widetilde{p}_+)}
%	=
%	(-1)^{(\ind p_\pm - \ind q_\pm)\cdot (\ind\widetilde{q}_- + \ind\widetilde{q}_+)}
$ as compared to na\"ive rearrangement. Similarly for $\general=\pfaff$ with
the index replaced by the dimension of the kernel modulo two. There is no sign
for $\general=\spectral$.
%Of course, every $\Z_2$\nobreakdash-torsor is canonically isomorphic to its dual, but not as a
%graded object. It is therefore non-trivial to include dualization in \eqref{eqn:excision-iso},
%and indeed $\general(p)$ has the expected algebraic properties only in the graded sense.
\end{rem}

\begin{rem}
Applying \eqref{eqn:excision-iso} for $\general=\spectral$ and $Y=S^1$
we deduce on the level of isomorphism classes an excision
formula for the spectral flow around a loop of self-adjoint
elliptic pseudo-differential operators on $X$:
\e\label{eqn:spectral-flow-cor}
	\SF\{Q^+_t\}_{t\in S^1}
	-\SF\{P^+_t\}_{t\in S^1}
	=
	\SF\{Q^-_t\}_{t\in S^1}
	-\SF\{P^-_t\}_{t\in S^1}
\e
For first order elliptic differential operators this is already known also a consequence of
the Atiyah--Patodi--Singer index theorem \cite[Th.~3.10]{AtPaSi1} which allows
one to express the spectral flow of a periodic family as an index on $X\times S^1$,
see \cite[p.~95]{AtPaSi3}. This depends on a detailed analysis of elliptic boundary
value problems using heat kernels. Once the spectral flow has been expressed
as an index, one can apply the classical excision formula \eqref{classical-excision} to get \eqref{eqn:spectral-flow-cor}.
Although this direct proof for \eqref{eqn:spectral-flow-cor} is compelling, it does not
strictly require the categorical point of view. A key point is that \eqref{eqn:excision-iso}
includes a generalization of \eqref{eqn:spectral-flow-cor} for non-periodic families that is
not expressable in terms of spectral flow.
\end{rem}

\subsubsection{Specialization for gauge theory}

For the convenience of the reader we now formulate Theorem~\ref{main-theorem}
in the special case relevant to gauge theory.

Let $G$ be a Lie group, $X$ a compact manifold, $P \to X$ a principal
$G$\nobreakdash-bundle, and $\Ad(P) \to X$ the associated bundle of Lie algebras.
Let
\[
	D\colon C^\infty(X,E^0) \longrightarrow C^\infty(X,E^1)
\]
be an elliptic differential operator on $X$, which is also denoted $E^\bu = (E^0, E^1, D)$.
Given fixed connections on $E^0, E^1, T^*X$, we can
use connections $\nabla^P$ on $P$ to define as in \cite[Def.~1.2]{JTU} the
\emph{$\nabla^P$\nobreakdash-twisted differential operator}
\[
D^{\nabla_{\Ad P}}\colon C^\infty\left(X,E^0\otimes \Ad(P)\right) \longrightarrow C^\infty\left(X,E^1\otimes \Ad(P)\right).
\]
Let $O^{E^\bu}_P = \orient(D^{\nabla_{\Ad P}})$ be its set of orientations.
More generally, a $Y$\nobreakdash-family of connections determines a double cover of $Y$.
Using
the trivial bundle $\underline{G}$ we define the \emph{normalized
orientation torsor} $\check{O}^{E^\bu}_P \coloneqq (O^{E^\bu}_P)^* \otimes_{\Z_2} O^{E^\bu}_{\underline{G}}$.

\begin{thm}\label{thm:gauge-special}
Let $X_\pm$ be compact manifolds. The data consisting of
\begin{enumerate}
\item[\textup{(a)}]
open covers $X_\pm = U_\pm \cup V_\pm$,
\item[\textup{(b)}]
principal $G$\nobreakdash-bundles $P_\pm \to X_\pm$ and
$G$\nobreakdash-frames $\tau_\pm$ of $P_\pm|_{V_\pm}$ over $V_\pm$
\item[\textup{(c)}]
elliptic operators $E^\bu_\pm = (E^0_\pm, E^1_\pm, D_\pm)$
%$D_\pm \colon C^\infty(X_-, E_\pm^0) \to C^\infty(X_+, E_\pm^1)$
on $X_\pm$,
\item[\textup{(d)}]
bundle isomorphisms $\Phi^\bu \colon E_-^\bu|_{U_-} \to E_+^\bu|_{U_+}$ covering a
diffeomorphism ${\phi\colon U_- \to U_+}$ identifying $D_\pm$ in these sense that
\[
\forall s\in C_\mathrm{cpt}^\infty(U_-,E_-^0):\quad \Phi^1\circ D_-(s) = D_+(\Phi^0\circ s\circ \phi^{-1}) \circ \phi,
\]
\item[\textup{(e)}]
an isomorphism $\Psi\colon P_-|_{U_-} \to P_+|_{U_+}$ covering $\phi$ of principal $G$\nobreakdash-bundles
satisfying $\Psi\circ \tau_- = \tau_+\circ \phi$ outside a compact subset of $U_-\cap V_- \cap \phi^{-1}(V_+)$,
\end{enumerate}
induces a canonical excision isomorphism of normalized orientation torsors
\e\label{gauge-excision}
	\Ex_{-+} \colon \check{O}^{E^\bu_-}_{P_-}
	\longrightarrow
	\check{O}^{E^\bu_+}_{P_+}.
\e
These have the following properties:
\begin{enumerate}

\item[\textup{(i)}]\textup{(Empty set.)}~If $U_\pm = \emptyset$ then we have canonical identifications
$\check{O}^{E^\bu_\pm}_{P_\pm} = \Z_2$ under which
$\Ex_{-+}$ becomes $\id_{\Z_2}$.

\item[\textup{(ii)}]\textup{(Restriction.)}~Assume $(\phi, \Phi^\bu)$ can be extended over
open supersets $U_- \subset \widetilde{U}_-$ and $U_+ \subset \widetilde{U}_+$ to $(\widetilde{\phi},\widetilde{\Phi}^\bu)$. Then
$\Ex_{-+} = \widetilde{\Ex}_{-+}$.

\item[\textup{(iii)}]\textup{(Sums.)}~In addition to \textup{(a)}--\textup{(e)} let $H$ be a Lie group,
$Q_\pm \to X_\pm$ principal $H$\nobreakdash-bundles, $\rho_\pm$ frames of $Q_\pm|_{V_\pm}$, and $\Xi \colon Q_-|_{U_-} \to Q_+|_{U_+}$ an isomorphism of $H$\nobreakdash-bundles over $\phi$ with $\Xi \circ \rho_- = \rho_+ \circ \phi$ outside a compact subset of $U_-\cap V_- \cap \phi^{-1}(V_+)$.
Then we have a commutative diagram
\[\xymatrix{
\check{O}^{E^\bu_-}_{P_-}\otimes\check{O}^{E^\bu_-}_{Q_-}\ar[r]_-{\eqref{eqn:symbols-direct-sums}}\ar[d]_{\Ex_{-+}^P\otimes \Ex_{-+}^Q}
&
\check{O}^{E^\bu_-}_{P_-\times_{X_-} Q_-}\ar[d]^{\Ex_{-+}^{P\times Q}}\\
\check{O}^{E^\bu_+}_{P_+}\otimes\check{O}^{E^\bu_+}_{Q_+}\ar[r]^-{\eqref{eqn:symbols-direct-sums}}
&
\check{O}^{E^\bu_+}_{P_+\times_{X_+} Q_+}.
}\]
As compared to na\"ive rearrangement, this diagram includes the sign given by
the parity of ${\dim H \cdot (\ind D_+ - \ind D_-)\cdot (\ind D_+^{\nabla_{\Ad P_+}} - \ind D_+^{\nabla_{\Ad G}})}$.

\item[\textup{(iv)}]\textup{(Functoriality.)}~Given three sets of data
$X_{\pm0} = U_{\pm0} \cup V_{\pm0}$,
$P_{\pm0} \to X_{\pm0}$,
$E^\bu_{\pm0}$,
 $\tau_{\pm0}$ as above, diffeomorphisms
 $\phi_{-0}\colon U_- \to U_0$, $\phi_{0+}\colon U_0 \to U_+$
 that identify $D_-, D_0, D_+$,
 and $G$\nobreakdash-bundle isomorphisms $\Psi_{-0}, \Psi_{0+}$
we have
\[
	\Ex_{0+}\circ \Ex_{-0} = \Ex_{-+}.
\]
Both \textup{(ii)} and \textup{(iii)} are natural for this functoriality.

\item[\textup{(v)}]\textup{(Global excision.)}~If $\phi\colon X_- \to X_+$ is a global diffeomorphism, $U_\pm = X_\pm$, then excision coincides with the global functoriality defined by mapping the kernels of the
differential operators $D^{\nabla_{\Ad P_\pm}}$ and $D^{\nabla_{\Ad \underline{G}}}$ using $\phi$ and $\Phi^\bu$.
\item[\textup{(vi)}]\textup{(Families.)}~Let $Y$ be compact Hausdorff. Given a $Y$\nobreakdash-family of data
as above, where $U_\pm, V_\pm \subset X_\pm \times Y$ and all the other data
are allowed to change in $Y$, \eqref{gauge-excision} becomes a continuous map of coverings over $Y$.
\end{enumerate}
\end{thm}

\begin{proof}
This is Theorem~\ref{main-theorem} applied to
the twisted principal symbols $p_{\pm,\xi}=\sigma_\xi(D_\pm)\otimes \id_{\Ad P_\pm}$ and
$q_{\pm,\xi}=\sigma_\xi(D_\pm)\otimes \id_{\Ad G}$.
\end{proof}

\subsection{Proof of Theorem~\ref{main-theorem}}\label{ssec:proofMAIN}

Assuming the categorical index calculus of Theorem~\ref{thm:symbol-calculus}, we
shall perform a series of reductions until Theorem~\ref{main-theorem} reduces completely
to Theorem~\ref{thm:symbol-calculus}, verifying in each step that
the properties claimed in Theorem~\ref{main-theorem} are preserved.

\subsubsection{Reduction to zeroth order}

Recall from Proposition~\ref{prop:red-of-order}(ii) that an
identification $\Xi\colon p \to q$ induces an identification
of the zeroth-order symbols. Assume Theorem~\ref{main-theorem}
in the special case of zeroth-order families.
Let $(p_\pm, \Xi_\pm, q_\pm)$ be two relative $Y$\nobreakdash-pairs,
compactly supported in $L_\pm \subset U_\pm$ and let
$\xymatrix@C=10ex@1{\ar[r]|{(\phi,\Pi,\Kappa)}(p_-, \Xi_-, q_-)|_{U_-}&
(p_+, \Xi_+, q_+)|_{U_+}}$ be an isomorphism of type {\ITEMgeneral} over a
diffeomorphism $\phi \colon U_- \to U_+$.
Using the excision isomorphism for zeroth-order families, we define $\Ex(\phi,\Pi,\Kappa)$ in general as
\begin{align*}
\general (p_-)^* \otimes \general (q_-) &\xrightarrow{\eqref{eqn:symbols-reduction-of-order}^{*,-1} \otimes \eqref{eqn:symbols-reduction-of-order}}
\general (p_-(p_-^\dagger p_-)^{-1/2})^* \otimes \general (q_-(q_-^\dagger q_-)^{-1/2})\\
&\xrightarrow{\Ex(\phi,\Pi,\Kappa)} \general (p_+(p_+^\dagger p_+)^{-1/2})^* \otimes \general (q_+(q_+^\dagger q_+)^{-1/2})\\
&\xleftarrow{\eqref{eqn:symbols-reduction-of-order}^{*,-1} \otimes \eqref{eqn:symbols-reduction-of-order}}\general (p_+)^* \otimes \general (q_+).
\end{align*}
Then Theorem~\ref{main-theorem}(vii) holds by definition and properties (i)--(vi)
follow from the compatibilities stated in Theorem~\ref{thm:symbol-calculus}(v) and
the assumed properties for zeroth-order families.

\subsubsection{Deformation to compactly supported symbols}

The reduction to Theorem~\ref{thm:symbol-calculus}
is based on the following deformation:

\begin{prop}\label{prop:deformation-of-rel-pairs}
In the notation of Definition~\textup{\ref{dfn:relative-pair}},
let $(p, \Xi, q)$ be a relative $Y$\nobreakdash-pair of order zero with compact support $L$
over a manifold $X$. Let $U \subset X$ be an open set with $L\subset U$, and
pick ${\chi \in C_\mathrm{cpt}^\infty(U, [0,1])}$ with $\chi|_L \equiv 1$.
%such that $\Xi^\bu$ are defined outside $\chi^{-1}(1)$.
Then % $\pi$ is the cotangent projection
\[
(p,\Xi,q)_{\chi}^t = \begin{pmatrix}
-(1-t+t\chi) p^\dagger & t(1-\chi) \pi^*{\Xi^0}^\dagger\\
t(1-\chi) \pi^*\Xi^1 & (1-t+t\chi) q
\end{pmatrix}
\colon
	\pi^*(E^1 \oplus F^0) \longrightarrow \pi^*(E^0 \oplus F^1)
\]
% *** WITHOUT t ***
%\[
%	(p,\Xi,q)_{\chi} \coloneqq
%	\begin{pmatrix}
%	-\chi p^\dagger & (1-\chi) \pi^*(\Xi^0)^\dagger\\
%	(1-\chi) \pi^*\Xi^1 & \chi q
%	\end{pmatrix}
%	\colon
%	\pi^*(E^1 \oplus F^0) \to \pi^*(E^0 \oplus F^1)
%\]
%*** WITH ROWS SWAPPED ***
%\[
%(p,\Xi,q)_{\chi} \coloneqq \begin{pmatrix}
%(1-\chi) \pi^*\Xi^1 & \chi q\\
%-\chi p^\dagger & (1-\chi) \pi^*(\Xi^0)^\dagger
%\end{pmatrix}
%\colon
%\pi^*(E^1 \oplus F^0) \to
%\pi^*(F^1 \oplus E^0)
%\]
for $t\in [0,1]$ has the following properties:
\begin{enumerate}
\item[\textup{(i)}]
For each $t \in [0,1]$, $(p,\Xi,q)_{\chi}^t$ is a zeroth-order family of elliptic symbols.
\item[\textup{(ii)}]
$(p,\Xi,q)_{\chi}^0 = -p^\dagger \oplus q$.
\item[\textup{(iii)}]
$(p,\Xi,q)_{\chi}^1$ has compactly support $\chi^{-1}(1) \subset U$.
\item[\textup{(iv)}]
When $p, q$ are skew-adjoint symbol families and $\Xi^0 = -\Xi^1$, all $(p,\Xi,q)_{\chi}^t$ are skew-adjoint. Similarly when $p,q$ are self-adjoint and $\Xi^0 = \Xi^1$.
\item[\textup{(v)}]
Let $\xymatrix@1@C=10ex{\ar[r]|{(\phi, \Pi, \Kappa)} (p_-,\Xi_-, q_-) & (p_+,\Xi_+, q_+)}$ be an isomorphism of relative pairs with support $L_\pm$ over the open embedding $\phi\colon U_- \to U_+$.
Let $\chi_\pm \in C^\infty_\mathrm{cpt}(U_\pm, [0,1])$ be cut-offs with $\chi_\pm|_{L_\pm}\equiv 1$.
Provided ${\chi_+ \circ \phi = \chi_-}$,
\e\label{eqn:compressed-identification}
	(\phi, \Pi^1\oplus\Kappa^0, \Pi^0\oplus \Kappa^1)\colon (p_-,\Xi_-,q_-)_{\chi_-}^t \longrightarrow (p_+,\Xi_+,q_+)_{\chi_+}^t
\e
is an identification over $\phi$ for each $t\in [0,1]$. This definition is clearly functorial for the composition of
relative pairs.
\end{enumerate}
\end{prop}

\begin{proof}
%This deformation is taken from \cite{DoKr}.
Note that $(p,\Xi,q)_{\chi}(t)$ is well-defined since $\chi|_L \equiv 1$ and $\Xi$ is defined
outside $L$. On $\chi^{-1}(1)$ we have $(p,\Xi,q)_{\chi}(t) = -p^\dagger \oplus q$ and (i) is clear.
To prove (i) for $x\notin \chi^{-1}(1)$, let $0\neq \xi \in T^*_xX$ and write
\[
{(p,\Xi,q)_{\chi}^t}(\xi)
=
\begin{pmatrix}
0& (\Xi^0)^\dagger\\
\Xi^1 & 0
\end{pmatrix}
\circ \left[
t(1-\chi) \id + 
(1-t+t\chi)\cdot
\begin{pmatrix}
0 & (\Xi^1)^\dagger q_\xi\\
-\Xi^0p_\xi^\dagger & 0
\end{pmatrix}
\right].
\]
Since by assumption $(-\Xi^0p_\xi^\dagger)^\dagger = - (\Xi^1)^{-1}q_\xi$, the latter summand is skew-adjoint and invertible. Hence all of its spectral values are non-zero
and purely imaginary. It follows that the endomorphism given by the inner square brackets does
not have zero in its spectrum.
% only $1-\chi(x) + \chi(x)\lambda$ with $\lambda$ in the spectrum of the skew-adjoint, inv.
The rest are trivial verifications.
\end{proof}

This deformation $(p,\Xi,q)_\chi^t$ determines the \emph{compression isomorphism}
\e
\label{eqn:compression}
	\general (p)^* \otimes \general (q) \xrightarrow{\eqref{eqn:symbols-direct-sums}, \eqref{eqn:symbols-adjoints}} \general (-p^\dagger \oplus q) \xrightarrow{\eqref{eqn:symbol-deformations}} \general \left(i^*(p,\Xi,q)_{\chi}^1\right).
\e
If $\chi_0, \chi_1, \ldots, \chi_r$ are cut-offs as in Proposition~\ref{prop:deformation-of-rel-pairs},
then so are all convex combinations
$\sum_{i=0}^r s_i\chi_i$ for $s \in \Delta^r$.
For $r=1$ we thus have a commutative diagram
\e\label{eqn:supp-def-triangle}
\begin{aligned}\xymatrix@R=1ex@C=5ex{
	&\general \left(i^*(p,\Xi,q)_{\chi_0}^1\right)\ar[dd]^{\general \{(p,\Xi,q)_{(1-s)\chi_0 + s\chi_1}\}_{s\in [0,1]}}\\
\general (p)^* \otimes \general(q)\ar[ru]\ar[rd] &\\
	&\general \left(i^*(p,\Xi,q)_{\chi_1}^1\right).
}\end{aligned}
\e

\subsubsection{Proof of Theorem~\ref{main-theorem} for zeroth-order families}

We must define $\Ex(\phi,\Pi,\Kappa)$ and verify Theorem~\ref{main-theorem} in the $m=0$ case.

Let $(\phi,\Pi,\Kappa)\colon (p_-,\Xi_-,q_-) \to (p_+,\Xi_+, q_+)$ be an isomorphism
of relative $Y$\nobreakdash-pairs of zeroth order.
Pick $\chi_+ \in C^\infty_\mathrm{cpt}(U_+,[0,1])$ with $\chi_+|_{L_+}\equiv 1$
and let $\chi_- \coloneqq \chi_+ \circ \phi$. Depending on this choice, we have by
Proposition~\ref{prop:deformation-of-rel-pairs} a deformation
$(p_\pm,\Xi_\pm,q_\pm)_\chi^t$ through symbols of
type {\ITEMgeneral}, beginning with $-p_\pm^\dagger \oplus q_\pm$ and
ending with a family $(p_\pm, \Xi_\pm, q_\pm)_{\chi_\pm}^1$ of elliptic symbols compactly
supported in $U_\pm$ and two compression isomorphisms \eqref{eqn:compression}.
Moreover, the isomorphism $(\phi,\Pi,\Kappa)$
induces identifications ${(\phi, \Pi^1\oplus\Kappa^0, \Pi^0\oplus \Kappa^1)}$ as in \eqref{eqn:compressed-identification}.
%\[
%(\phi,\Pi^1\oplus \Kappa^0, \Pi^0\oplus \Kappa^1)\colon (p_-,\Xi_-,q_-)_{\chi_-}^t
%\longrightarrow
%(p_+,\Xi_+,q_+)_{\chi_+}^t,\quad \forall t\in[0,1].
%\]
Define
\e\label{eqn:zeroth-order-excision-def}
\begin{aligned}
\xymatrix@C=18ex{
\general (p_-)^* \otimes \general (q_-)\ar@{-->}[r]_{\Ex(\phi,\Pi,\Kappa)}
\ar[d]_{\eqref{eqn:compression}}
&
\general (p_+)^* \otimes \general (q_+)
\ar[d]^{\eqref{eqn:compression}}\\
\general \left( (p_-, \Xi_-, q_-)_{\chi_-}^1 \right)
&
\general \left( (p_+, \Xi_+, q_+)_{\chi_+}^1 \right)\\
\general \left( (p_-, \Xi_-, q_-)_{\chi_-}^1|_{U_-} \right)\ar[u]_{i^-_!}^{\eqref{eqn:pushforwards}}
&
\general \left((p_+, \Xi_+, q_+)_{\chi_+}^1|_{U_+} \right)\ar[u]^{i^+_!}_{\eqref{eqn:pushforwards}}
\ar[l]^{\eqref{eqn:maps-induced-by-identification}}_-{(\phi,\Pi^1\oplus \Kappa^0, \Pi^0\oplus \Kappa^1)^*}
}
\end{aligned}\e
The composition $\Ex(\phi,\Pi,\Kappa)$ is independent of the choice of $\chi_+$ by
\eqref{eqn:supp-def-triangle}, since according to Corollary~\ref{cor:deformations} both
\eqref{eqn:maps-induced-by-identification}, \eqref{eqn:pushforwards} are compatible
with deformations.

Property (i) then follows by using $\chi_+ = 0$ and similarly (ii) follows by using the same cut-off
for the supersets. Since all of the maps used in the construction of \eqref{eqn:zeroth-order-excision-def}
are functorial and compatible with direct sums and deformations, we see that (iii)--(vi)
follow from the corresponding properties in Theorem~\ref{thm:symbol-calculus}.

\section{Determinant, Pfaffian, and spectral covers}

\subsection{Sign convention and spectral preliminaries}\label{ssec:sign-conventions}

\subsubsection{Supersymmetric sign convention}

The top exterior power `$\Lambdatop$' of a finite-dimensional vector space has the
property that a short exact sequence
$\Sigma\colon 0\to U \xrightarrow{i} V \xrightarrow{p} W \to 0$ induces
an isomorphism
\e\label{equation.3.1}
\begin{aligned}
\operatorname{det}_\Sigma\colon
\Lambdatop U \otimes \Lambdatop W \longrightarrow \Lambdatop V,\qquad
u\otimes p_*(w) \mapsto i_*(u)\wedge w,
%&u_1 \wedge \ldots \wedge u_k
%\otimes
%p(v_1) \wedge \ldots \wedge p(v_\ell)
%\\
%&\mapsto i(u_1)\wedge \ldots \wedge i(u_k)\wedge
%	v_1\wedge \ldots \wedge v_\ell.
\end{aligned}
\e
where $u \in \Lambda^{\dim U} U$, $w \in \Lambda^{\dim W} V$, and $i_*, p_*$
are the maps induced by $i, p$ on exterior powers.
This expresses our orientation convention that in such a sequence
$U$\nobreakdash-coordinates are regarded to come before $W$\nobreakdash-coordinates in $V$.
The data \eqref{equation.3.1} defines a determinant functor, in the sense
of Deligne~\cite[{\paragraphS4.3}]{Del}, which has a unique extension
to bounded complexes \cite{Knu}, subject to a sign convention.

We use the sign convention of supersymmetry, where
vectors and co-vectors are viewed as odd and scalars
as even when commuting them. This rule determines
various isomorphisms involving tensor
products and dualization, which for convenience we make
explicit. Thus we regard the determinant as a $\Z_2$\nobreakdash-graded line in
degree $\equiv \dim V \in \Z_2$, use the graded tensor
product, and braid
\e\label{eqn:tau}
\sigma\colon L_1 \otimes L_2 \longrightarrow L_2\otimes L_1,
\quad
\sigma(x_1\otimes x_2) = (-1)^{\deg L_1\cdot \deg L_2} x_2 \otimes x_1.
\e
We agree to evaluate functionals on the left
${\operatorname{ev}\colon L \otimes L^* \to \R}$, $(x,\al)\mapsto \langle x,\al\rangle = \al(x)$,
so that evaluation on the right introduces a sign $(-1)^{\deg L}$. This convention
matches \eqref{eqn:detP} in that the dual appears there also on the right.
Instead of the na\"ive one, we insist on the identification
\e\label{L1L2dual}
	\tau\colon L_1^*\otimes L_2^* \longrightarrow (L_2\otimes L_1)^*,
	\quad
	(x_2\otimes x_1)(\al_1\otimes \al_2)\coloneqq
	\langle x_1, \al_1\rangle \langle x_2, \al_2\rangle.
\e
In the same way, to identify $\Lambdatop(V^*)$ with $(\Lambdatop V)^*$
we must use
\[
	\mathcal{P}\colon
	\al_1 \wedge \ldots \wedge\al_n \longmapsto
	\left(
	v_n \wedge \ldots \wedge v_1 \longmapsto \operatorname{det}[v_j(\al_i)]_{i,j=1}^n
	\right),
\]
which differs by $(-1)^{n(n-1)/2}$ from the na\"ive convention.
%The evaluation
%pairing determines a dual basis map $\mathcal{D}\colon L\setminus \{0\} \to L^*\setminus \{0\}$ with $\operatorname{ev}(x\otimes \mathcal{D}x)=1$.
For the dual of an exact sequence
$\Sigma^*\colon 0 \to W^* \to V^* \to U^* \to 0$ we then get
commutative diagrams
\[
\begin{aligned}
\xymatrix@C=10ex{
	\Lambdatop (W^*) \otimes \Lambdatop(U^*)
	\ar[d]_{(\mathcal{P}\otimes \mathcal{P})}
	\ar[r]_-{\det_{\Sigma^*}} & \Lambdatop(V^*)\ar[d]^{\mathcal{P}}\\
(\Lambdatop W)^*\otimes (\Lambdatop U)^*\ar[r]^-{(\det_\Sigma^{-1})^*\circ\tau} & (\Lambdatop V)^*.
}
\end{aligned}
\]

\subsubsection{Spectral theory of Fredholm operators}

%The extension of  the determinant to bounded finite-dimensional complexes turns
%out to depend only on the cohomology, see Knudsen~\cite{Knu}. This allows the
%following generalization to Fredholm operators over the real or the complex numbers.
We shall use a definition of the Quillen determinant for Hilbert spaces due
to Bismut--Freed \cite[Sect.~f)]{BiFr} in terms of spectral theory.

Recall that the \emph{essential spectrum} of a bounded operator is defined as its
spectrum in the Calkin algebra $\mathcal{B}/\mathcal{K}$
modulo compact operators. By definition, $0\notin \spec_\mathrm{ess}(A_0)$
for a Fredholm operator $A_0$. As the spectrum in $\mathcal{B}/\mathcal{K}$ is
a closed set, we can then find an essential spectral gap $(-\de,\de)$ with $\de>0$.

\begin{lem}[see~\cite{Phi}]\label{lem:saFred}
For a self-adjoint Fredholm operator ${A_0\colon H \to H}$ let $\de > 0$ be
such that $\spec_\mathrm{ess}(A_0)\cap (-\de, \de)=\emptyset$.
%That is, assume each $A_0 - \lambda$ for $-\de < \lambda < \de$ is Fredholm.
For all $ -\de < \nu < 0 < \mu < \de$ with $\mu,\nu \notin \spec A_0$ there exists a
neighborhood $U$ of $A_0$ in the space of self-adjoint Fredholm operators with the following properties:
\begin{enumerate}
\item[\textup{(i)}]
$\forall A \in U: \nu, \mu \notin \spec A$ and $\spec_\mathrm{ess}(A) \cap (-\de,\de) = \emptyset$.
\item[\textup{(ii)}]
The direct sum $V_{(\nu,\mu)}(A) \subset H$ of all eigenspaces of $A$ with
eigenvalue $\nu < \lambda < \mu$ defines a vector bundle over $U$ of finite locally
constant rank.
\end{enumerate}
\end{lem}

\begin{proof}
Pick $\varepsilon > 0$ such that both $(\nu-\varepsilon,\nu+\varepsilon)$ and $(\mu-\varepsilon, \mu+\varepsilon)$ are disjoint from $\spec A_0$.
This remains true in a neighborhood $U$ of $A_0$ where we
require also $\spec_\mathrm{ess}(A) \cap (-\de,\de) = \emptyset$.
Then all $A-\lambda$ with $\nu < \lambda < \mu$ are Fredholm and $A$ has discrete
eigenvalues near zero. Hence $V_{(\nu,\mu)}(A)$ is finite-dimensional.
The projection $\chi_{(\nu,\mu)}(A)$ onto $V_{(\nu,\mu)}(A)$ can be formed using
functional calculus.
Let $f$ be the continuous function with $f|_{[\nu,\mu]}\equiv 1$,
$f_{(-\infty,\nu-\varepsilon]}\equiv 0$, $f_{[\mu+\varepsilon,\infty)}\equiv 0$,
and that is otherwise affine-linear. Then $\chi_{(\nu,\mu)}(A) = f(A)$ for all $A \in U$
and the map $A\mapsto f(A)$ from $U$ into the bounded projections is continuous. In particular
it has locally constant finite rank, so its image is a vector subbundle.
\end{proof}

\begin{dfn}\label{dfn:iso-of-operators}
Let $\cH^0, \cH^1$ be Hilbert bundles over a space $Y$. The bounded operators
in each fiber define a Banach bundle $\mathcal{B}(\cH^0,\cH^1)$ in the operator norm.
A \emph{$Y$\nobreakdash-family of Fredholm operators} is a continuous section $P\colon Y \to \mathcal{B}(\cH^0,\cH^1)$ that is fiberwise Fredholm.
Let $P_\pm \colon Y \to \mathcal{B}(\cH^0_\pm,\cH^1_\pm)$ be families of Fredholm operators.
An \emph{isomorphism} $F^\bu\colon P_- \to P_+$ is a pair of continuous sections ${F^\bu\colon Y \to \mathcal{B}(\cH^\bu_-,\cH^\bu_+)}$ of invertible operators satisfying $F^1\circ P_- = P_+ \circ F^0$.
\end{dfn}

\subsection{Determinant cover of real Fredholm operators}\label{ssec:Determinant}

\subsubsection{Determinant line bundle}

\begin{dfn}
	Let $H^0, H^1$ be Hilbert spaces.
	The \emph{determinant line} of a Fredholm
	operator $P \colon H^0 \to H^1$ (regarded as a two term complex) is
	\e\label{eqn:detP}
	\det P \coloneqq \Lambdatop \Ker P \otimes \left(\Lambdatop \Ker P^\dagger\right)^*,
	\e
	a $\Z_2$\nobreakdash-graded line in degree $(-1)^{\ind P}$. More generally, for a Fredholm
	operator between Banach spaces, replace $\Ker P^\dagger$ by $\Coker P$.
\end{dfn}

For a family of Fredholm operators the disjoint union over all
\eqref{eqn:detP} will be topologized as a line bundle using
Lemma~\ref{lem:saFred}. It generalizes that of Freed~\cite{Fre} for Dirac operators.
One may alternatively use `stabilization' to topologize the determinant line bundle
for general Banach spaces, see for example Zinger~\cite{Zing}, paying attention
to a tedious sign convention, as in~\eqref{eqn:detSigma}.

We say $\mu > 0$ is \emph{sufficiently small} for $P$
if $[0,\mu] \cap \spec_\mathrm{ess}(P^\dagger P) = \emptyset$.
% Then $[0,\mu]$ is also disjoint from $\spec_\mathrm{ess}(P P^\dagger)$, because
% in any ring the non-zero spectrum of a*b equals the non-zero specgrum of b*a.
For $P$ Fredholm and $0< \mu < \nu$ sufficiently small, define
\[
\begin{aligned}
\Lambdatop V_{[0,\nu]}(P^\dagger P) \otimes \left(\Lambdatop V_{[0,\nu]}(PP^\dagger)\right)^*
&\xrightarrow{\stab_{\nu,\mu}}
\Lambdatop V_{[0,\mu]}(P^\dagger P) \otimes \left(\Lambdatop V_{[0,\mu]}(PP^\dagger)\right)^*
\\
\stab_{\nu,\mu}(v \wedge w \otimes
	\be \wedge \al)
	&=
	\be(\Lambda(P) w) \cdot v\otimes\al,
%	v_k(\al_k) \cdots v_1(\al_1) v\otimes \al,
\end{aligned}
\]
for
$v \in \Lambdatop V_{[0,\mu]}(P^\dagger P), \al \in \left(\Lambdatop V_{[0,\mu]}(PP^\dagger)\right)^*, w \in \Lambdatop V_{(\mu,\nu]}(P^\dagger P)$, and $\be \in \left(\Lambdatop V_{(\mu,\nu]}(PP^\dagger)\right)^*$.
As $P$ restricts to an isomorphism ${P \colon V_\lambda(P^\dagger P) \to V_\lambda(PP^\dagger)}$
when $\lambda \neq 0$, the top exterior power makes sense on $(\mu,\nu]$. It is easy to check
\e\label{stab-compatible}
	\stab_{\mu,0}\circ\stab_{\nu,\mu}=\stab_{\nu,0}.
\e

\begin{dfn}
Let $P$ be a $Y$\nobreakdash-family of Fredholm operators.
The {\it determinant line bundle} of $P$ is
the $Y$\nobreakdash-family of one-dimensional vector spaces $\{\det P_y\}_{y\in Y}$ with
the following topology.
Let $y_0 \in Y$ and pick $\mu>0$ sufficiently small for $P_{y_0}$ with
$\mu \notin \spec P^\dagger_{y_0}P_{y_0}$.
Pick a neighborhood $U$ of $y_0$ over which the Hilbert bundles are trivial. By shrinking $U$
we may also assume the conclusions of Lemma~\ref{lem:saFred}. Then we can transport the bundle topology provided by Lemma~\ref{lem:saFred}(ii) on exterior powers along the isomorphisms
\e\label{eqn:topology-det}
\stab_{\mu,0} \colon
\Lambdatop V_{[0,\mu]}(P^\dagger P)\otimes \left( \Lambdatop V_{[0,\mu]}(PP^\dagger) \right)^*
\longrightarrow \det P|_U.
\e
This topology is independent of $\mu$ by \eqref{stab-compatible} and since $\operatorname{stab}_{\nu,\mu}$ are homeomorphisms.
% More generally, let $V \to Y$ be a vector bundle. Then $\Lambdatop V \to Y$ gets a
% topology of line bundle and $(\Lambdatop V)^*$ can also be topologized as a line bundle
% in such a way that evaluation is continuous. The point is therefore that $V_{(\mu,\nu]}(P^\dagger P)$
% and $V_{(\mu,\nu]}(PP^\dagger)$ are vector bundles and that ${P^{-1} \colon V_\lambda(PP^\dagger) \to V_\lambda(P^\dagger P)}$ is continuous when $\lambda \neq 0$, hence induces a continuous map on exterior powers.
It is also easily seen to be independent of the trivializations of the
Hilbert bundles on overlaps. Hence we get a line bundle $\det P \to Y$.
\end{dfn}

\begin{prop}\label{prop:properties-det}
The determinant line bundle has the following properties:
\begin{enumerate}
\item[\textup{(i)}] \textup{(Functoriality.)}~Let $P_\pm$ be $Y$\nobreakdash-families of Fredholm operators.
An isomorphism $F^\bu\colon P_- \to P_+$ induces an isomorphism
\e\label{eqn:det-functorial}
\Lambdatop F^0 \otimes \left(\Lambdatop F^1\right)^{*,-1}
\colon
	\det P_- \longrightarrow \det P_+.
\e
% In fact it is not necesssary for $F$ to be metric preserving
\item[\textup{(ii)}] \textup{(Direct sums.)}~For $Y$\nobreakdash-families of Fredholm operators $P, Q$ there is a canonical isomorphism, natural for \eqref{eqn:det-functorial},
\e\label{eqn:det-direct-sums}
	\operatorname{det}_{P,Q}\colon \det P \otimes \det Q \longrightarrow \det(P\oplus Q).
\e
These are associative.
They are graded commutative in the sense of a commutative diagram
\e\label{eqn:det-graded-commutative}
\begin{aligned}
\xymatrix@C=7ex{
		\det P \otimes \det Q\ar[d]_\sigma\ar[r]_-{\operatorname{det}_{P,Q}} & \det (P\oplus Q)\ar[d]^{\det(\operatorname{swap})}\\
		\det Q \otimes \det P\ar[r]^-{\operatorname{det}_{Q,P}} & \det (Q\oplus P).
}
\end{aligned}
\e
Here $\sigma$ includes $(-1)^{\ind P \cdot \ind Q}$ and the isomorphism $\operatorname{swap}$ exchanges
the Hilbert spaces without a sign. More generally, a short exact sequence
\[
	\Sigma\colon 0 \longrightarrow P_0 \longrightarrow P_1 \longrightarrow P_2 \longrightarrow 0
\]
of Fredholm operators, meaning a diagram of bounded operators
\e\label{eqn:KES-of-Fredholm}
\begin{aligned}
\xymatrix{
	0\ar[r] & H_-^0 \ar[r]_{i^-}\ar[d]_{P^0} & H_-^1\ar[d]_{P^1}\ar[r]_{p^-} & H_-^2\ar[d]_{P^2}\ar[r] & 0\\
	0\ar[r] & H_+^0 \ar[r]^{i^+} & H_+^1\ar[r]^{p^+} & H_+^2\ar[r] & 0
}
\end{aligned}
\e
with exact rows and $P^0, P^1, P^2$, induces an isomorphism
\[
	\operatorname{det}_\Sigma \colon \det P^0 \otimes \det P^2 \longrightarrow \det P^1.
\]
\item[\textup{(iii)}] \textup{(Adjoints.)}~There is a canonical isomorphism $\tau\colon \det P^\dagger \to (\det P)^*$.
It is natural for \eqref{eqn:det-functorial}\textup{:}
\e\label{eqn:adjoint-natural}
\xymatrix{
\det P^\dagger\ar[r]_-{\tau} & (\det P)^*\\
\det Q^\dagger\ar[u]^{\det f^\dagger}\ar[r]^-{\tau} & (\det Q)^*.\ar[u]_{(\det f)^*}
}
\e
For the exact sequence $\Sigma^\dagger\colon 0 \to P_2^\dagger \to P_1^\dagger \to P_0^\dagger \to 0$ adjoint to \eqref{eqn:KES-of-Fredholm} and the isomorphisms we have a commutative diagram
\e\label{eqn:adjoint-KES}
\xymatrix@R=3ex{
\det P_2^\dagger \otimes \det P_0^\dagger \ar[d]_{\tau\otimes\tau}\ar[r]_-{\operatorname{det}_{\Sigma^\dagger}} & \det P_1^\dagger\ar[dd]^\tau\\
(\det P_2)^* \otimes \det (P_0)^*\ar[d]_{\eqref{L1L2dual}} & \\
(\det P_0 \otimes \det P_2)^* & (\det P_1)^*\ar[l]_-{(\operatorname{det}_\Sigma)^*}
}
\e
\item[\textup{(iv)}] \textup{(Composition.)}~When ${P = P^+ \circ P^-}$ factors into two Fredholm operators
${P^-\colon H^- \to H^0}$ and ${P^+ \colon H^0 \to H^+}$,
we get an isomorphism
\[
	\operatorname{det}_{P^-,P^+}\colon \det P^- \otimes \det P^+ \longrightarrow \det P.
\]
\item[\textup{(v)}] \textup{(Invertible.)}~For a $Y$\nobreakdash-family of invertible operators $P$
the determinant line bundle has a canonical continuous trivialization $\operatorname{det}(P) = 1\otimes 1^*$.
\end{enumerate}
\end{prop}

\begin{proof}
(i)\hskip\labelsep
Assume $F$ is metric preserving.
The induced isomorphism is continuous, since $F$ induces a map
on the left hand side of \eqref{eqn:topology-det}.
In fact, by passing to $F(F^\dagger F)^{-1/2}$ one need not assume $F$
to be metric preserving. Alternatively, everything
makes sense for Fredholm operators between Banach spaces.

\smallskip\noindent
(ii)\hskip\labelsep
A short exact sequence \eqref{eqn:KES-of-Fredholm} determines a snake lemma exact sequence\\
\begin{minipage}{\textwidth}
\begin{tikzpicture}
	\matrix (m) [matrix of math nodes,row sep=2em,column sep=-1em]
	{
		\Ker P^0 && \Ker P^1 && \Ker P^2 && \Coker P^0 && \Coker P^1 && \Coker P^2\\
		&\phantom{\Coker F}&& \Coker i^-_* && \Coker p^-_* && \Coker \delta &&\phantom{\Coker F}\\
	};

	\node [below = 0ex of m-1-1] {$a$};
	\node [below = -1ex of m-2-4] {$d$};
	\node [below = -1ex of m-2-6] {$e,\ga$};
	\node [below = -1ex of m-2-8] {$\be$};
	\node [below = -1ex of m-1-11] {$\ze$};

	\path[right hook-latex] (m-1-1) edge node[above]{\small$i^-_*$} (m-1-3);
	\path[-latex] (m-1-3) edge node[above]{\small$p^-_*$} (m-1-5);
	\path[-latex] (m-1-5) edge node[above]{\small$\delta$} (m-1-7);
	\path[-latex] (m-1-7) edge node[above]{\small$i^+_*$} (m-1-9);
	\path[>=latex,->>] (m-1-9) edge node[above]{\small$p^+_*$} (m-1-11);
	
	%	\path[>=latex,->>] (m-1-1) edge node[anchor=east]{\small$i^0$} (m-2-2);
	\path[>=latex,->>] (m-1-3) edge (m-2-4);
	\path[>=latex,->>] (m-1-5) edge (m-2-6);
	\path[>=latex,->>] (m-1-7) edge (m-2-8);
%	\path[>=latex,->>] (m-1-9) edge (m-2-10);
	
	%	\path[right hook-latex] (m-2-2) edge (m-1-3);
	\path[right hook-latex] (m-2-4) edge node[below right]{\small$p^-_*$} (m-1-5);
	\path[right hook-latex] (m-2-6) edge node[below right]{\small$\delta$}(m-1-7);
	\path[right hook-latex] (m-2-8) edge node[below right]{\small$i^+_*$} (m-1-9);
\end{tikzpicture}
\end{minipage}
that we splice as indicated into short exact sequences $\Sigma_{\Ker P^1}$,
$\Sigma_{\Ker P^2}$, $\Sigma_{\Coker P^0}$, $\Sigma_{\Coker P^1}$ named
after their middle term. Using \eqref{equation.3.1}, we define, denoting the dual homomorphisms
by $(\cdot)^*$,
\e\label{eqn:detSigma}
\begin{aligned}
\operatorname{det}_\Sigma\colon &\det P^0 \otimes \det P^2 \longrightarrow \det P^1,\\
&a \otimes (\operatorname{det}^{-1}_{\Sigma_{\Coker P^0}})^*(\be\otimes \ga) \otimes
\operatorname{det}_{\Sigma_{\Ker P^2}}(d\otimes e) \otimes \zeta\\
&\longmapsto
\ep \cdot \ga(e)\cdot \operatorname{det}_{\Sigma_{\Ker P^1}}(a\otimes d)\otimes (\operatorname{det}^{-1}_{\Sigma_{\Coker P_1}})^*(\ze\otimes \be),
\end{aligned}
\e
using the sign $\ep = (-1)^{\ga+\ga d + \be(d+\ze)}$ dictated by
our supersymmetry conventions. This sign convention
agrees with that of Zinger~\cite[(4.10)]{Zing}. We refer
to \cite[(2.27) and Cor.~4.13]{Zing} for the tedious verification of associativity.
For a direct sum
the coboundary $\delta$ vanishes and \eqref{eqn:detSigma} reduces to
\begin{align*}
\hskip-1ex\det P^0 \otimes \det P^2
=&
\Lambdatop \Ker P^0 \otimes \left(\Lambdatop \Coker P^0\right)^*
\otimes
\Lambdatop \Ker P^2 \otimes \left(\Lambdatop \Coker P^2\right)^*\\
\overset{(1\otimes 1\otimes \tau)(1\otimes\sigma\otimes 1)}{=}&
\Lambdatop \Ker P^0\otimes\Lambdatop \Ker P^2
\otimes
\left(\Lambdatop \Coker P^2 \otimes \Lambdatop \Coker P^0\right)^*\\
\overset{1\otimes 1\otimes \sigma^*}{=}&
\Lambdatop \Ker P^0\otimes\Lambdatop \Ker P^2
\otimes
\left(\Lambdatop \Coker P^0 \otimes \Lambdatop \Coker P^2\right)^*,
\end{align*}
and so to na\"ive rearrangement with sign $(-1)^{\dim \Coker P^0\cdot \ind P^2}$.
In this case, which is all we need, the associativity and graded commutativity is easy.

\smallskip
\noindent
(iii)\hskip\labelsep
We define $\tau$ to be \eqref{L1L2dual}.
The verifications of \eqref{eqn:adjoint-natural} and \eqref{eqn:adjoint-KES} are straightforward,
inserting signs whenever commuting symbols and using that the snake lemma exact sequence
of $\Sigma^\dagger$ is the adjoint of the snake lemma sequence of $\Sigma$.

\smallskip
\noindent
(iv)\hskip\labelsep
This follows by applying \eqref{equation.3.1} to the splicings of the exact sequence
\[
	0\to \Ker P^- \to \Ker P  \to \Ker P^+ \to \Coker P^- \to \Coker P \to \Coker P^+ \to 0,
\]
% Splicing: Insert the groups \Ker P^+ \cap \Im P^-, \Ker P^+ / (\Ker P^+ \cap \Im P^-), H^0 / (\Ker P^+ + \Bild P^-)
inserting again signs when rearranging. One may also regard
(iv) as a special case of (ii) using the short exact sequence ${\Sigma: 0 \to P^- \to P\oplus \id_{H^0} \to P^+ \to 0}$ and then define $\operatorname{det}_{P^-,P^+}$ as $\operatorname{det}_\Sigma$ followed by $\det(P\oplus \id_{H^0}) \cong \det(P)\otimes \det (\id_{H^0}) \cong \det (P)$, again using (ii).
%\[\xymatrix@C=10ex{
%	0\ar[r] & H^-\ar[d]_{P^-} \ar[r]^-{(\id,P^-)} & H^- \oplus H^0 \ar[d]_{P\oplus \id} \ar[r]^-{P^-\oplus -\id} & H^0\ar[r]\ar[d]_{P^+} & 0\\
%	0\ar[r] & H^0\ar[r]^-{(P^+,\id)} & H^+\oplus H^0 \ar[r]^-{\id\oplus -P^+} & H^+\ar[r] & 0.
%}\]

\noindent
(v)\hskip\labelsep
Obviously \eqref{eqn:detP} is canonically trivial and $\operatorname{det}(P)\coloneqq 1\otimes 1^* \in \det(P)$ is a continuous section, by \eqref{eqn:topology-det} for $\mu$ smaller than
the least eigenvalue of $P^\dagger P$.
\end{proof}

\subsubsection{Determinant cover of real Fredholm operators}

\begin{dfn}
The {\it determinant cover} of a $Y$\nobreakdash-family $P$ of \emph{real} Fredholm operators is the
principal $\Z_2$\nobreakdash-bundle $\orient P \coloneqq \left(\det P \setminus \{\text{zero section}\}\right) / \R_{>0}$ of $Y$,
regarded as being $\Z_2$\nobreakdash-graded in degree $(-1)^{\ind P}$.
\end{dfn}

\begin{rem}
For $P$ real or complex self-adjoint Fredholm, evaluation defines a canonical trivialization of \eqref{eqn:detP}.
For real operators we thus get a canonical basepoint $o_{\mathrm{taut},P} \in \orient P$.
However, in families we still get non-trivial determinant covers, since they vary discontinuously.
To understand the discontinuity, suppose for simplicity that $P_t$ is a family
with a single eigenvalue $\lambda_t$ in $[-\mu,\mu]$
which crosses zero once and upwards at $t_0$. There is a corresponding continuous eigenvector
$v_t$. By definition of the topology, $\orient P_t$ is the set of orientations of the vector space of endomorphisms of $\Lambdatop V_{[0,\mu^2]}P^2_t \cong \R$ (identified using $v_t$).
For $t\neq t_0$ the tautological element is represented by the identity of
$\Ker P_t = \{0\}$, viewed as an endomorphism of $V_{[0,\mu^2]}P^2_t$ via $P_t$.
In terms of this identification, $o_{\mathrm{taut},P_t}$ becomes the orientation
of the multiplication map $\lambda_t\colon \R \to \R$. At $t=t_0$ we take the
orientation of the identity map of $\Ker P_{t_0}=\R$. It follows that
\[
o_{\mathrm{taut},P_t}=
\begin{cases}
  -1 & (t<t_0),\\
  +1 & (t\geq t_0).
\end{cases}
\]
This agrees with the parity of the spectral flow along $\{P_t\}$, see~\cite{Phi}. We therefore get
a continuous section $(-1)^{\SF\{ P_s \}_{s \in [0,t]}}\cdot o_{\mathrm{taut},P_t}$
of $\orient(P)$.

In the self-adjoint case we can thus represent $\orient(P)$ by pairs $[\mu,\ep]$
of $\mu > 0$ with $\pm \mu \notin \spec P$ and $\ep \in \{\pm 1\}$. Here
$[\mu,\ep] = [\nu, \de]$ when $\de = \ep\cdot (-1)^{N_{]\mu,\nu[}(P)}$.
% Say for for $0<\mu<\nu$ 
\end{rem}

\begin{prop}
The properties \textup{(i)}--\textup{(v)} of Proposition~\textup{\ref{prop:properties-det}} hold analogously
for the determinant cover. The monodromy of the determinant cover around a loop of Fredholm operators
is the parity of the spectral flow.
\end{prop}

\subsection{Pfaffian cover of real skew-adjoint Fredholm operators}\label{ssec:Pfaffian}
\subsubsection{Pfaffian line bundle}

Initially we can work over the real or the complex numbers.

\begin{dfn}
	Let $P \colon H \to H$ be a skew-adjoint Fredholm operator.
	The \emph{Pfaffian line} of $P$ is the graded line in degree $(-1)^{\dim \Ker P}$ defined by
	\[
		\Pfaff P \coloneqq \Lambdatop \Ker P.
	\]
\end{dfn}

The  inner product defines a two-form $\langle P\cdot , \cdot \rangle$ that is 
non-degenerate on each finite-dimensional eigenspace $V_{\lambda}(P^\dagger P)$.
This induces a preferred volume element $\omega_\lambda \in \Lambdatop V_{\lambda}(P^\dagger P)$.
For $0 < \mu < \nu$ sufficiently small for $P$ define a map
$\stab_{\nu,\mu}^\mathrm{skew}\colon
	\Lambdatop V_{[0,\nu]}(P^\dagger P)
	\to
	\Lambdatop V_{[0,\mu]}(P^\dagger P)$
by
${\stab_{\nu,\mu}^\mathrm{skew}
(v\wedge \omega_{\lambda_1}\wedge\ldots \wedge \omega_{\lambda_k})
=v}$,
where $v \in \Lambdatop V_{[0,\mu]}(P^\dagger P)$ and where ${\mu < \lambda_1 < \cdots < \lambda_k} \leq \nu$ are all non-zero eigenvalues of $P^\dagger P$ in this range. Then
\[
	\stab_{\mu,0}^\mathrm{skew}\circ \stab_{\nu,\mu}^\mathrm{skew} = \stab_{\nu,0}^\mathrm{skew}
\]

\begin{dfn}
Let $P$ be a family of skew-adjoint Fredholm operators.
For $\nu, \mu \notin \spec(P^\dagger P)$ sufficiently small for $P$
the maps $\stab_{\nu,\mu}^\mathrm{skew}$ are homeomorphisms, using the topology
provided by Lemma~\ref{lem:saFred}. As before, we get a well-defined
topology on $\Pfaff P$ by transporting the topology along
\e\label{eqn:topology-pfaff}
\stab_{\mu,0}^\mathrm{skew}
\colon
\Lambdatop V_{[0,\mu]}(P^\dagger P)
\longmapsto
\Lambdatop \Ker P.
\e
\end{dfn}

\begin{prop}\label{prop:properties-pfaff}
The Pfaffian has the following properties:
\begin{enumerate}
\item[\textup{(i)}] \textup{(Functoriality.)}~Let $P_\pm$ be $Y$\nobreakdash-families of skew-adjoint
Fredholm operators. A metric-preserving isomorphism $F^\bu \colon P_- \to P_+$ induces an
isomorphism
\e\label{eqn:pfaffian-functoriality}
	\Lambdatop F^0\colon
	\Pfaff P_- \longrightarrow \Pfaff P_+.
\e
\item[\textup{(ii)}] \textup{(Direct sums.)}~For $Y$\nobreakdash-families of skew-adjoint Fredholm operators $P, Q$ there is
a canonical isomorphism, natural for \eqref{eqn:pfaffian-functoriality},
\e\label{eqn:pfaff-sums}
	\pfaff_{P,Q}\colon \Pfaff P \otimes \Pfaff Q \longrightarrow \Pfaff (P\oplus Q).
\e
These are graded commutative as in \eqref{eqn:det-graded-commutative}.
\item[\textup{(iii)}] \textup{(Root.)}~For a $Y$\nobreakdash-family of skew-adjoint operators there is a
canonical isomorphism, functorial for isomorphisms as in \textup{(i)} with $F^0=F^1$,
\[
	\Pfaff P\otimes \Pfaff P \longrightarrow \det P.
\]
\item[\textup{(iv)}] \textup{(Invertible.)}~For a $Y$\nobreakdash-family of invertible skew-adjoint Fredholm
operators the Pfaffian line bundle has a canonical trivialization.
\end{enumerate}
\end{prop}

\begin{proof}
(i)
$F^0$ restricts to a map $V_\lambda(P_-^\dagger P_-)\to V_\lambda(P_+^\dagger P_+)$
and $F^1$ restricts to a map $V_\lambda(P_- P_-^\dagger)\to V_\lambda(P_+ P_+^\dagger)$.
This is because $F^1(P_- P_-^\dagger) (F^1)^{-1} = P_+ P_+^\dagger$
and $F^0(P_-^\dagger P_-) (F^0)^{-1} = P_+^\dagger P_+$.
Hence $F^0, F^1$ can be used to extend $\Lambdatop F^0$ to
the left-hand side of \eqref{eqn:topology-pfaff}, which is therefore continuous.
The condition $F^0 = F^1$ is required for functoriality in (iv). The remaining assertions
are obvious.
\end{proof}

\subsubsection{Pfaffian cover of real skew-adjoint Fredholm operators}

\begin{dfn}
The \emph{Pfaffian cover} of a family $P$ of real skew-adjoint Fredholm operators
is the principal $\Z_2$\nobreakdash-bundle $\pfaff P \coloneqq (\Pfaff P\setminus \{\text{zero section}\})/\R_{>0}$,
regarded as being $\Z_2$\nobreakdash-graded in degree $(-1)^{\dim_\R \Ker P}$.
\end{dfn}

Properties \textup{(i)}--\textup{(iv)} of Proposition~\textup{\ref{prop:properties-pfaff}} hold also for the Pfaffian cover.
In particular, for real skew-adjoint elliptic symbol families the orientation cover
is canonically trivial, since $\orient P = \pfaff P \otimes_{\Z_2} \pfaff P$ by (iii) and
since the square of any double cover is canonically trivial.

\begin{rem}
Since the eigenspaces for small non-zero eigenvalues $\la$ have the symplectic form $\omega_\la$,
they all have even multiplicity. It follows that the spectral flow around a loop of skew-adjoint
Fredholm operators is always even.
%The monodromy of the Pfaffian cover around a loop of skew-adjoint operators is
%the parity of half of its spectral flow.
\end{rem}

\subsection{Spectral cover of self-adjoint Fredholm operators}
\label{ssec:spectral-cover}

The constructions in the section are taken over the complex numbers, but
they apply equally to real operators by complexifying,
since $\pi_1 \Fred^\mathrm{sa}_\R = \pi_1 \Fred^\mathrm{sa}_\C$.

\subsubsection{Construction}

For a self-adjoint operator $P$ and bounded $I\subset \R\setminus \spec_\mathrm{ess}(P)$
let $N_I(P)$ be the number of eigenvalues $\lambda \in \spec P$ in $I$,
counted with multiplicity. Similarly, sums or products over eigenvalues $\lambda$ are always
taken with multiplicity.

\begin{dfn}\label{dfn:spectral-torsor}
Let $P$ be a self-adjoint Fredholm operator. An element of the \emph{spectral torsor} $\spectral(P)$ is
represented by a pair $(\mu,n)$ of $\mu > 0$ and $n\in \Z$, where $\pm \mu \notin \spec(P)$
and $[-\mu,\mu]\cap \spec_\mathrm{ess}(P) = \emptyset$.
Here we regard $(\mu, n) \sim (\nu, m)$ for $0<\mu<\nu$ as equivalent if $m-n = N_{]\mu,\nu[}(P)$.
\end{dfn}

\begin{dfn}
Let $P$ be a $Y$\nobreakdash-family of self-adjoint Fredholm operators. The disjoint union
over all $\spectral(P_y)$ is topologized as follows.
Let $y_0 \in Y$ and pick $\delta > 0$ with
${\spec_\mathrm{ess}(P_{y_0}) \cap (-\delta,\delta) = \emptyset}$.
Suppose $\delta > \mu > 0$, $\pm \mu \notin \spec P_{y_0}$. As
in Lemma~\textup{\ref{lem:saFred}} this remains true in a neighborhood $U$ of $y_0$.
For each $n \in \Z$ the map $y\mapsto [\mu,n]$ is, by definition, a continuous section
$U\to \spectral(P)|_U$.
\end{dfn}

The following properties are fairly obvious from the definition:

\begin{prop}\label{prop:properties-spectral}
The spectral cover has the following properties:
\begin{enumerate}
\item[\textup{(i)}] \textup{(Functoriality.)}~Let $P_\pm$ be $Y$\nobreakdash-families of
self-adjoint Fredholm operators. A self-adjoint isomorphism
$F^\bu \colon P_- \to P_+$ induces an isomorphism
\e\label{eqn:spectral-functorial}
	\spectral P_- \longrightarrow \spectral P_+.
\e
\item[\textup{(ii)}] \textup{(Direct sums.)}~For $Y$\nobreakdash-families of self-adjoint Fredholm operators $P, Q$ there
is a canonical isomorphism
\e\label{eqn:spectral-direct-sums}
	\spectral(P)\otimes \spectral(Q) \longrightarrow \spectral(P \oplus Q).
\e
These are natural for \eqref{eqn:spectral-functorial} and commutative
as in \eqref{eqn:det-graded-commutative} without sign.
\item[\textup{(iii)}] \textup{(Negative.)}~There is a canonical isomorphism $\spectral(-P) \to \spectral(P)^*$,
natural for \eqref{eqn:spectral-functorial} and compatible with \eqref{eqn:spectral-direct-sums}.
\item[\textup{(iv)}] \textup{(Invertible.)}~Let $P$ be a family of invertible self-adjoint operators.
Then $\spectral(P)$ is canonically trivial. In particular, for positive definite families.
\item[\textup{(v)}] The monodromy of the covering $\spectral(P)$ is the spectral flow around loops.
\end{enumerate}
\end{prop}

\begin{rem}
Let $P$ be a first order elliptic differential operator of Dirac type.
Then the spectral torsor can also be obtained by a reduction of
$\xi(P)\coloneqq \frac12(\eta(P,0) + \dim \Ker P)$ modulo integers.
Recall here that
\[
	\eta(P,s) \coloneqq \sum_{0\neq \lambda \in \spec(P)}
	\frac{\operatorname{sgn}(\lambda)}{|\lambda|^s},
	\qquad
	\Re\mathrm{e}(s) > \dim X,
\]
is the $\eta$\nobreakdash-invariant of $P$, continued analytically to $s=0$ as in Atiyah--Patodi--Singer~\cite{AtPaSi1}. A canonical isomorphism of $\Z$\nobreakdash-torsors is given by
\[
	\spectral(P)
	\longrightarrow
	\xi(P) + \Z\subset \R,
	\quad
	[\mu,n]
	\longmapsto 
	\xi(P) + n - N_{[0,\mu[}(P).
\]
It is continuous, since both $\xi(P)$, $N_{[0,\mu[}(P)$ jump according to the spectral flow.
\end{rem}

\subsubsection{Orientations graded over $\Z_n$}\label{gradedZn}

Let $n\in \N$, $P$ be a $Y$\nobreakdash-family of self-adjoint Fredholm operators, and
$\sigma$ an $\Z_n$\nobreakdash-orientation of this family, meaning a section
$\sigma \colon Y \to \spectral(P)\otimes_\Z \Z_n$. Suppose each $P_y$ is invertible,
which is the expected generic situation under perturbation.
For example, $Y$ could be a moduli space of connections that satisfy
a non-linear elliptic equation whose linearization is self-adjoint.

We have two elements of $\spectral(P_y)\otimes_\Z \Z_n$, namely $\sigma(y)$
and the trivialization of Proposition~\ref{prop:properties-spectral}(iv) and these
differ by $\mu(y) \in \Z_n$, which may be used as weights for counting.
When $Y$ is finite such numbers $p=\mu(y)$ can, given also a finite set of
`preferred' trajectories, be used to define chain groups $C^p$ graded over $p\in \Z_n$
in the style of Floer \cite{Floe}.

\begin{ex}
Let $P \to S^7$ be an $\SU(4)$\nobreakdash-bundle. Let $D$ be the real Diracian on $S^7$,
a self-adjoint operator. Let $D^{\nabla_{\Ad P}}$ be the Dirac operator twisted
by a connection $\nabla_P \in \mathcal{A}_P$ on $P$.
The spectral cover of this $\mathcal{A}_P$\nobreakdash-family is then a $\Z$\nobreakdash-cover of $\mathcal{A}_P$,
which is trivial since the base space is contractible. Any trivialization descends to
a trivialization of $\spectral(D^{\nabla_{\Ad P}})\otimes \Z_8$ over the
configuration space of connections modulo gauge. This is because every gauge
transformation acts by a multiple of $8$. To see this, let $g \in \Aut(P)$ and
pick a path of connections $\nabla^t_{P}$ from $\nabla_P$ to $g^*\nabla_P$.
The mapping torus bundle $Q \to S^7 \times S^1$, obtained by identifying endpoints
of $P\times [0,1]$ using $\phi$, can be used to calculate the spectral flow
$\SF\{D^{\nabla^t_{\Ad P}}\}_{t \in [0,1]}$ as in Walpuski~\cite{Walp3}, based
on Atiyah--Patodi--Singer~\cite{AtPaSi3}:
\[
	\ind D^{\Ad Q}
	=
	\int_{S^7\times S^1}
	\frac{5}{3} c_2(Q)^2 - \frac{4}{3}c_4(Q) + \frac{1}{3} c_2(Q)p_1(S^7)
	=
	-\frac{4}{3}\int_{S^7\times S^1} c_4(Q).
\]
The Euler number $\int_{S^7\times S^1} c_4(Q)$ of any $\SU(4)$\nobreakdash-bundle $Q \to S^7\times S^1$
is divisible by six.
This can be seen as follows. Represent $Q$ by gluing two trivial $\SU(4)$\nobreakdash-bundles along the
equator $S^7 \times \{1/2\}$ using a map $f\colon S^7 \to \SU(4)$.
By counting zeros of a section constructed from $f$
%
%By counting zeros of
%the section $s(x,t)=(1-t)\vec{e}_1 + t f(x) \vec{e}_1$ of $Q\times_{\SU(4)} \C^4$ we see that
the Euler number of $Q$ is seen to be
equals to the degree of $f$ composed with the projection ${\pi\colon \SU(4) \to \SU(4)/\SU(3)=S^7}$.
But from the long exact sequence of homotopy groups of a fibration
\[\xymatrix{
	\pi_7(\SU(4))\ar[r]^{\pi_*}
	&
	\pi_7(S^7)=\Z\ar[r]^-{\partial}
	&
	\pi_6(\SU(3))=\Z_6\ar[r]
	&
	\pi_6(\SU(4))=0
}\]
the image of $\pi_*$ is $6\Z$, so the image of $(\pi\circ f)_* = \pi_* \circ f_*$ is also
divisible by $6$.
\end{ex}

\subsection{Transgression and the spectral cover}

We now establish a relation between spectral torsors and the
orientation cover.

\begin{dfn}
The \emph{transgression} of a complex line bundle ${L \to X}$ is the
principal $\Z$\nobreakdash-bundle ${S\to \Map(S^1,X)}$ whose fibers
\[
	S_\ga \coloneqq \left\{  \tau\colon S^1 \longrightarrow \gamma^*L\setminus \{0\} \right\} / \simeq,
	\qquad \forall \ga \in \Map(S^1,X),
\]
are the homotopy classes of trivializations along $\ga$. The $\Z$\nobreakdash-action
${e^{2\pi int}\tau(t)}$ is freely transitive, since any two $\tau$ differ by the mapping
degree of some ${S^1 \to \C^*}$.
To define the topology on $S$, let $\mathcal{U} \subset \Map(S^1, X)$ be an open set.
Then a continuous section $\tau\colon S^1 \times \mathcal{U} \to L\setminus \{0\}$
with $\tau(t,\ga) \in L_{\ga(t)}$ determines, by definition, a continuous
section $\mathcal{U} \to S, \ga \mapsto [\tau(\cdot,\ga)]$ of the covering $S$.
\end{dfn}

\begin{lem}\label{lem:Det-mu-trivial}
For a self-adjoint Fredholm operator $A_0\colon H \to H$, let 
$\delta > \mu > 0$, $\pm \mu \notin \spec A_0$, $\nu=-\mu$, and $U$ be a neighborhood
of $A_0$ in the self-adjoint Fredholm operators as in Lemma~\textup{\ref{lem:saFred}}.
Then the restriction of the determinant line bundle to $U$ is trivial.
\end{lem}

\begin{proof}
Consider the $Y=U\times [0,1]$\nobreakdash-family $A_t\coloneqq A - t\mu$.
By Lemma~\ref{prop:properties-det}(v) the determinant line bundle of this family
is canonically trivial at $t=1$ and the restrictions of the bundle to either endpoints
of $[0,1]$ are isomorphic.
\end{proof}

\begin{thm}\label{thm:sa-generalFred}
Under the equivalence of Atiyah--Patodi--Singer \textup{\cite{AtSiSkew}}, \textup{\cite{AtPaSi3}},
\[
\Fred^\mathrm{sa}_\C \overset{\simeq}{\longrightarrow} \Omega \Fred_\C,\quad
P \longmapsto P_t\coloneqq\begin{cases}
\cos(t)+i\sin(t)P & (0\leq t \leq \pi),\\
\cos(t)+i\sin(t) & (\pi \leq t \leq 2\pi),
\end{cases}
\]
we have that the spectral cover $\spectral(P)$ is canonically isomorphic to the transgression ${S \to \Omega \Fred_\C}$
of the complex determinant line bundle $\det P_t$.% a $\Z$-principal bundle over $\Omega \Fred_\C$.
\end{thm}

Being homotopy equivalent, any two connected $\Z$\nobreakdash-coverings of
$\Fred^\mathrm{sa}_\C$ are clearly isomorphic, but a canonical isomorphism
fixes an integer choice here.

\begin{proof}
Recall as in Lemma~\ref{lem:saFred} that near zero all eigenvalues of a self-adjoint
Fredholm operator $P$ are discrete with finite-dimensional eigenspaces. By the
spectral theorem, the eigenvalues of the normal operator $P_t^\dagger P_t$ for $t \in [0,\pi]$
are ${\lambda(t)\coloneqq\cos^2(t) + \sin^2(t)\lambda^2}$, where $\lambda$ ranges over the spectrum of $P$.
In particular, $P_t$ is invertible
unless $t=\pi/2$. Moreover, for $|\lambda|$ small we have
\e\label{eigenspaces-Pt}
	V_{\cos^2(t) + \sin^2(t)\lambda^2}\left(P_t^\dagger P_t\right) = V_{\lambda^2}(P^2) = V_{\lambda}(P) \oplus V_{-\lambda}(P).
\e
(one of the spaces on the right can be trivial.)
This determines a map $\phi_t(\mu)$
\begin{align*}
\det P_t
&\overset{\eqref{eqn:topology-det}}{\longrightarrow}
\Lambdatop V_{[0,\mu(t)[}(P_t^\dagger P_t)
\otimes
\left(\Lambdatop V_{[0,\mu(t)[}(P_tP_t^\dagger)\right)^*\\
&\overset{\eqref{eigenspaces-Pt}}{=}\Lambdatop V_{[0,\mu[}(P^\dagger P)
\otimes \left(\Lambdatop V_{[0,\mu[}(PP^\dagger)\right)^*\overset{\eqref{eqn:topology-det}}{\longrightarrow} \det P
\end{align*}
for $t \in [0,\pi]$ and $\mu>0$ sufficiently small, $\pm \mu \notin \spec P$. Since $P_t$
is invertible,
$\operatorname{det} P_t \coloneqq 1\otimes 1^* \in \det P_t$ is
a canonical trivialization. Unravelling the definition of \eqref{eqn:topology-det}, it is
easy to check the formulas
\begin{align}
\phi_t(\nu) &=
	\prod_{
	\begingroup\def\arraystretch{0.5}
	\begin{array}{c}
	\scriptstyle
	\lambda \in \spec P\\
	\scriptstyle
	\mu<|\lambda|<\nu
	\end{array}
	\endgroup
	}(\lambda^{-1}\cos t + i\sin t) \cdot \phi_t(\mu),\nonumber\\
\phi_t(\mu) \left( \operatorname{det} P_t \right) &= 
\prod_{
	\begingroup\def\arraystretch{0.5}
	\begin{array}{c}
	\scriptstyle
	\lambda \in \spec P\\
	\scriptstyle
	0<|\lambda|<\mu
	\end{array}
	\endgroup
	}
	\frac{\lambda^{-1}\cos(t) + i \sin(t)}{\lambda^{-1}\cos(s) + i \sin(s)}
	\cdot
	\phi_s(\mu) \left( \operatorname{det} P_s \right).\label{phimu2}
	% Only for $t,s \neq \pi/2$.
\end{align}
By deforming, any element of the transgression $S_{t\mapsto P_t}$ can be represented by
a map ${t\mapsto \tau(t) \in \det P_t}$ with $\tau(t) = \operatorname{det} P_t$
for $t\in [0,2\pi] \setminus ]0,\pi[$. We can thus represent
these by maps
$\ga_\mu(t)=\phi_\mu(t) \circ \tau(t)$, $\ga_\mu\colon [0,\pi] \to \det P \setminus \{0\}$ satisfying by \eqref{phimu2} the periodicity condition $\ga_\mu(0) = (-1)^{N_{]-\mu,\mu[}}\cdot \ga_\mu(\pi)$, modulo homotopy through such maps.
By Lemma~\ref{lem:Det-mu-trivial} the choice of $\mu$ determines also a trivialization $\Phi\colon \C \to \det P$. In this trivialization we can describe $\ga_\mu$ up to homotopy by its mapping degree.
In other words, we can put any $\ga_\mu$ into standard form $e^{i \left( -N_{]-\mu,\mu[}(P) + 2n\right) t}$. The sought-for isomorphism is then
\[
	\spectral P
	\longrightarrow
	S_{t\mapsto P_t},
	\quad
	[n,\mu]
	\longmapsto
	\left[
		\phi_t(\mu)^{-1}\circ\Phi\left(e^{i \left( -N_{]-\mu,\mu[}(P) + 2n\right) t}\right)
	\right].
\]
Another choice of trivialization differs by multiplication by some $\ka \in \C\setminus \{0\}$
which does not effect the homotopy class on the right. It is independent of $\mu$ by \eqref{phimu2},
where we can use also the homotopic $\exp\left( \left[N_{]\mu,\nu[}(P) - N_{]-\nu,-\mu[}(P) \right] \cdot it\right)$ as prefactor.
According to Lemma~\ref{lem:Det-mu-trivial} we can perform all this over open
sets $U$, which proves continuity of the isomorphism there.
\end{proof}

\begin{rem}
Using this result one may alternatively reduce Theorem~\ref{main-theorem} for $\general = \spectral$ to
the case $\general = \orient$.
\end{rem}

\subsection{Proof of Theorem~\ref{thm:symbol-calculus}}\label{ssec:proof-calc}

\subsubsection{The covering of an elliptic symbol family}\label{sssec:covering-symbol}

Let $\general \in \{\orient, \pfaff, \spectral\}$. By Propositions
\ref{prop:properties-det}, \ref{prop:properties-pfaff}, and \ref{prop:properties-spectral}
we can assign to every $Y$\nobreakdash-family of Fredholm operators $P$ of type {\ITEMgeneral} a $\Z_2$\nobreakdash-graded
$G$\nobreakdash-principal-bundle $\general(P) \to Y$ with the degree and structure group $G$
given by \eqref{eqn:degree}.

\begin{dfn}
A \emph{deformation} $P=\{P_t\}_{t\in [0,1]}$ of type {\ITEMgeneral}
is a $Y\times [0,1]$\nobreakdash-family of Fredholm operators of type {\ITEMgeneral}.
Then $P$ determines a covering of $Y \times [0,1]$ in which
fiber transport gives an isomorphism
\e\label{eqn:orient-deformations}
	\general \{P_t\}_{t\in [0,1]}\colon
	\general (P_0)
	\longrightarrow
	\general (P_1).
\e
\end{dfn}

By general properties of covering maps, \eqref{eqn:orient-deformations}
depends on $\{P_t\}_{t\in [0,1]}$ only up to homotopy through operators
of type {\ITEMgeneral} relative endpoints $\{0,1\}$ and is functorial for
the juxtaposition of deformations. Moreover, a
$[0,1]$\nobreakdash-deformations of isomorphisms $F^\bu_t\colon P_t \to Q_t$
as in Definition~\textup{\ref{dfn:iso-of-operators}} induces a
commutative diagram
\[
\xymatrix{
\general(P_0)\ar[d]_{\general(F^\bu_0)}\ar[r]_{\eqref{eqn:orient-deformations}}&\general(P_1)\ar[d]^{\general(F^\bu_1)}\\
\general(Q_0)\ar[r]^{\eqref{eqn:orient-deformations}}&\general(Q_1)
}\]
using the functoriality defined in part (i) of Propositions \ref{prop:properties-det}, \ref{prop:properties-pfaff}, and
\ref{prop:properties-spectral}. Moreover, the naturality of parts (ii) and (iii)
shows that \eqref{eqn:orient-deformations} is compatible in the obvious sense with direct
sums and adjoints.

\begin{dfn}
Let $p \in \Ell_Y^m(X;E^0,E^1)$ be a 
$Y$\nobreakdash-family of elliptic symbols of type {\ITEMgeneral} over a manifold $X$.
Suppose $X$ is compact or that $m=0$ and $p$ is supported in some $L\times Y$ with $L \subset X$ compact.
The set $\sigma^{-1}(p)_{(\general)} \subset \PsiDO_Y^m(X; E^0, E^1)$ of
compactly supported $Y$\nobreakdash-families of $\PsiDO$s $P$ of type {\ITEMgeneral} with principal symbol $p$ is non-empty and
convex, by Theorem~\ref{thm:Fredholm-Properties}(iii). % For $U=X$.
By Lemma~\ref{lem:convex} the straight line between
$P_0, P_1 \in \sigma^{-1}(p)_{(\general)}$
is a $Y\times [0,1]$\nobreakdash-family $P$ in $\sigma^{-1}(p)_{(\general)}$ of compactly supported
operators.
Then \eqref{eqn:orient-deformations} defines an isomorphism
${t_{P_0, P_1}\colon \general (P_0) \to \general (P_1)}$. % depending only on $P_0$ and $P_1$.
For $P_0, P_1, P_2 \in \sigma^{-1}(p)_{(\general)}$
we have $t_{P_1, P_2} \circ t_{P_0, P_1} = t_{P_0, P_2}$.
The limit over these is the set of compatible families:
\[
	\general (p) \coloneqq
	\left\{
		\left(x_P \in \general (P)\right)_{P \in \sigma^{-1}(p)_{(\general)}}
		\;\middle|\;
		\forall P_0, P_1 \in \sigma^{-1}(p)_{(\general)}:
		t_{P_0, P_1} \left( x_{P_0} \right) = x_{P_1}
	\right\}.
\]
Of course, for any $P \in \sigma^{-1}(p)_{(\general)}$ the projection defines an isomorphism
\[
	\operatorname{ev}_P
	\colon
	\general(p) \longrightarrow \general(P).
\]
In particular, $\general(p)$ has a canonical topology of $G$\nobreakdash-principal bundle.
\end{dfn}

\subsubsection{General properties}

We now summarize the general properties of $\general(p)$.
These are straightforward consequences of Propositions~\ref{prop:properties-det}, \ref{prop:properties-pfaff}, and \ref{prop:properties-spectral}.
On the level of symbols, functoriality takes the following form.
An identification ${\Phi\colon p_- \to p_+}$ of type {\ITEMgeneral} over
a diffeomorphism ${\phi\colon X_- \to X_+}$
induces an isomorphism of coverings
\e\label{eqn:maps-induced-by-identificationX}
	\Phi^*\colon\general (p_-) \longrightarrow \general (p_+).
\e
When $X_\pm = \emptyset$ we take \eqref{eqn:maps-induced-by-identificationX} to mean the identity map.

To see \eqref{eqn:maps-induced-by-identificationX} choose $P_+ \in \sigma^{-1}(p_+)_{(\general)}$. Then $P_-\coloneqq (\phi,\Phi)^*P_+$ (see Definition~\ref{dfn:pullback-family}) represents $p_-$ and is of type {\ITEMgeneral}, using our type
assumption on $\Phi$. Moreover
$\Phi\colon P_- \to P_+$ is an isomorphism of Fredholm operators,
and \eqref{eqn:maps-induced-by-identificationX} is defined as
$\general(\Phi)$, meaning
\eqref{eqn:det-functorial}, \eqref{eqn:pfaffian-functoriality}, or
\eqref{eqn:spectral-functorial}, composed with the canonical projections.
Similarly we have, from the corresponding properties of Fredholm operators,
direct sum and adjointness isomorphisms, natural for \eqref{eqn:maps-induced-by-identificationX},
\begin{align}
\general(p)\otimes \general(q) &\longrightarrow \general(p\oplus q),\label{eqn:symbols-direct-sumsX}\\
\general(-p^\dagger) &\longrightarrow \general(p)^*.\label{eqn:symbols-adjointsX}
\end{align}
Recall then \eqref{eqn:symbol-deformations} as a consequence. Thus, for a deformation $\{p_z\}_{t\in [0,1]}$ we have
\e\label{eqn:symbol-deformationsX}
	\general\{p_z\}_{z \in [0,1]}\colon \general(p_0) \longrightarrow \general(p_1).
\e
\subsubsection{Proofs of Theorem~\ref{thm:symbol-calculus}(iv) and (v)}

We first state conditions ensuring the continuity of %the functoriality maps
\eqref{eqn:det-functorial}, \eqref{eqn:pfaffian-functoriality}, and \eqref{eqn:spectral-functorial}.
These are obvious from the definition of the topology.

\begin{enumerate}
\item[\ITEMorient]
Let $0<\delta_y<d\left(0,\spec_\mathrm{ess}\left(A_y^\pm(A_y^\pm)^\dagger\right)\right)$ be
lower semicontinuous for two $Y$\nobreakdash-families of real Fredholm operators $A^\pm$ such that
\[
S(y) \coloneqq \spec\left(A_y^-(A_y^-)^\dagger\right) \cap [0, \delta_y) = \spec \left(A_y^+(A_y^+)^\dagger\right) \cap [0, \delta_y)
% Then also the non-zero spectrum of $(A^\pm)^\dagger A^\pm$ agree.
\]
for all $y \in Y$. For each $y \in Y$ and $\lambda \in S(y) \cup \{0\}$ let
\begin{align*}
F_{\lambda,y}\colon V_\lambda\left(A_y^-(A_y^-)^\dagger\right)
\longrightarrow
V_\lambda\left(A_y^+(A_y^+)^\dagger\right)\\
G_{\lambda,y}\colon V_\lambda\left((A_y^-)^\dagger A_y^-\right)
\longrightarrow
V_\lambda\left((A_y^+)^\dagger A_y^+\right)
\end{align*}
be isomorphisms. Assume that these are continuous as follows.
For every $y_0 \in Y$ and $0<\mu < \delta_{y_0}$ not in $S(y_0)$
there is by Lemma~\textup{\ref{lem:saFred}} a neighborhood $V$ of $y_0$
for which $V_{[0,\mu]}\left(A^\pm (A^\pm)^\dagger\right)$, $V_{[0,\mu]}\left((A^\pm)^\dagger A^\pm \right)$ have vector bundle topologies over $V$ and we can 
even find $V$ with $\mu < \delta_y$ for all $y\in V$. % By the lower semicontinuity
For all such neighborhoods we
require $\bigoplus_{\lambda \in [0,\mu]} F_{\lambda,\mu}$
and $\bigoplus_{\lambda \in [0,\mu]} G_{\lambda,\mu}$ to
be continuous. % bundle homomorphisms over $V$.
Then the $Y$\nobreakdash-family of isomorphisms \eqref{eqn:det-functorial}
define an isomorphism $\det A^-_y \to \det A^+_y$
of line bundles over $Y$.
\item[\ITEMpfaff]
Given two real skew-adjoint families $A^\pm$, let $\delta_y$ be as in {\ITEMorient} and let
$F_{\lambda,y}\colon V_\lambda\left((A_y^-)^\dagger A_y^-\right)
\to
V_\lambda\left((A_y^+)^\dagger A_y^+\right)$
be an orthogonal isomorphism for all $y\in Y$ and $\lambda \in S(y)\cup \{0\}$,
continuous as in {\ITEMorient}.
Then the $Y$\nobreakdash-family of isomorphisms \eqref{eqn:pfaffian-functoriality}
is a continuous map ${\operatorname{Pf}(A^-_y) \to \operatorname{Pf}(A^+_y)}$ over $Y$.
\item[\ITEMspectral]
Given two complex self-adjoint families $A^\pm$, let
$0<\delta_y<d\left(0,\spec_\mathrm{ess}(A_y)\right)$ be lower semicontinuous
such that
\[
	S(y) \coloneqq \spec(A_y^-) \cap (-\delta_y,\delta_y) = \spec(A_y^+) \cap (-\delta_y,\delta_y),
	\enskip \forall y \in Y.
\]
For all $y\in Y, \lambda \in {S(y)\cup\{0\}}$ let
$F_{\lambda,y}\colon V_\lambda(A_y^-) \to V_\lambda(A_y^+)$
be isomorphisms,
continuous as in {\ITEMorient}. % Slightly different, but should be clear
Then \eqref{eqn:spectral-functorial} is a continuous isomorphism of coverings
$\spectral(A^-) \to \spectral(A^+)$.
\end{enumerate}

\begin{cor}
Let $\general \in \{\orient,\pfaff,\spectral\}$.
Let $P \in \PsiDO^0(X;E^0, E^1)$ be a $Y$\nobreakdash-family of zeroth-order $\PsiDO$s of type {\ITEMgeneral}, compactly supported in the image of an open embedding $i\colon U\hookrightarrow X$ and
with $\sigma(P)$ elliptic and compactly supported. Then extension by zero defines an continuous isomorphism
\e\label{pushforward-incl}
	i_!\colon \general(i^*P) \longrightarrow \general(P).
\e
We allow $U=\emptyset$ in which case $\general(i^*P)$ is trivial.
The maps \eqref{pushforward-incl} are functorial for \eqref{eqn:maps-induced-by-identificationX}
compatible with direct sums \eqref{eqn:symbols-direct-sumsX} and adjoints \eqref{eqn:symbols-adjointsX}.
Moreover, we have $i_! \circ j_! = (i\circ j)_!$ and $\id_!=\id$.
\end{cor}

\begin{proof}
This follows from Theorem~\ref{thm:Fredholm-Properties}(ii).
Property (i) is a restatement of \eqref{eqn:det-direct-sums}, \eqref{eqn:pfaff-sums}, and \eqref{eqn:spectral-direct-sums}, while (ii) is trivial.
\end{proof}

\begin{prop}
Let $\general \in \{\orient,\pfaff,\spectral\}$.
Let $p$ be a $Y$\nobreakdash-family of elliptic symbols of type {\ITEMgeneral}.
Then we have canonical \emph{reduction of order} isomorphisms
\e\label{eqn:symbols-reduction-of-orderX}
	\general(p) \longrightarrow \general\left( p(p^\dagger p)^{-1/2} \right).
\e
These are compatible with the isomorphisms \eqref{eqn:maps-induced-by-identificationX}--\eqref{eqn:symbol-deformationsX}.
\end{prop}

\begin{proof}
Choose a pseudo-differential operator $P$ representing $p$. Then $1+P^\dagger P$
is a positive definite operator and we may define $(1+P^\dagger P)^{-1/2}$ using
unbounded functional calculus.
One can construct a pseudo-differential operator $Q$ with the property that
$S\coloneqq (1+P^\dagger P)^{-1/2}-Q$ is a smoothing operator and $\sigma(Q)=(p^\dagger p)^{-1/2}$. Then
$P(1+P^\dagger P)^{-1/2}$ has the same kernel and cokernel as $P$ and
can be deformed along the straight line $P\circ [Q+tS]$ to $P\circ Q$, a pseudo-differential
operator representing $p(p^\dagger p)^{-1/2}$. This gives \eqref{eqn:symbols-reduction-of-orderX} for
$\general=\orient$.

We can assume $Q$ to be self-adjoint. % replace S and Q by 1/2(S+S^*) and 1/2(Q+Q^*).
When $P$ is self-adjoint or skew-adjoint, we can then perform the deformation through
operators of the same type. Moreover, the spectrum of the Fredholm operator $P(1+P^\dagger P)^{-1/2}$ near
zero can then be identified with the spectrum of the unbounded operator $P$ near zero.
Hence \eqref{eqn:symbols-reduction-of-orderX} follows from the stated continuity conditions
for $\general \in \{\pfaff,\spectral\}$.
\end{proof}

\begin{rem}
For $\general=\orient$ we also have a more formal proof.
Suppose $r$ is positive-definite. Write $r_{\xi,y}=q_{\xi,y}^\dagger q_{\xi,y}$ for
$0\neq \xi \in T_x^*X$, $y \in Y$ using the unique positive square root.
Then $q \in \Ell^{m/2}_Y(X;E,E)$. The adjointness and composition properties determine
a canonical trivialization
$\underline{\Z}_2 \cong (\orient q)^* \otimes \orient (q) \leftarrow \orient (q^\dagger) \otimes \orient (q) \to \orient (r)$ for $r$ positive-definite. Of course, this follows also from Proposition~\ref{prop:properties-det}(v).
Applied to $r=(p^\dagger p)^{-1/2}$, the composition property determines a canonical isomorphism
$\orient \left(p(p^\dagger p)^{-1/2}\right) \cong \orient(p) \otimes \orient (p^\dagger p)^{-1/2} \cong \orient(p)$.
\end{rem}

\appendix

\section{Compactly supported $\PsiDO$s}\label{sec-symbols-and-operators}

For non-compact manifolds there are various types
of pseudo-differential operators depending on the
regularity assumption at infinity. The rather restrictive class of
`compactly supported operators' suffice for our purpose.
Their theory reduces quickly to finitely many charts
as in the compact case.

\subsection{Families of pseudo-differential operators}\label{families}

Let $Y$ be a topological space. A {\it $Y$\nobreakdash-family of manifolds} is a space
$E$ with a continuous map $\pi\colon E \to Y$ and a collection of smooth
structures on each fiber $E_y=\pi^{-1}(y)$.
%An open set $U \subset X \times Y$ will be regarded as a $Y$-family using the projection.
%Generally, subscripts will indicate restriction to the fiber.
Let $(E_\pm,\pi_\pm)$ be $Y$\nobreakdash-families of manifolds.
A {\it $Y$\nobreakdash-family} of smooth maps is a continuous map $\phi\colon E_- \to E_+$
with $\pi_+\circ \phi = \pi_-$ whose restriction $\phi_y$
to every fiber is smooth. For a $Y$\nobreakdash-family of open embeddings, diffeomorphisms,
or bundle isomorphisms we additionally require $\phi$ to be a homeomorphism
onto its image and every $\phi_y$ to be of the corresponding type.\smallskip

We next outline some basic theory of pseudo-differential operators.
For open subsets $\Omega \subset \R^d$, H\"ormander~\cite{Hoer}
is a good reference. In this paper, we work with the following combination
of definitions from Atiyah--Singer~\cite[p.~509]{AtSi1}, \cite[p.~123]{AtSi4},
\cite[p.~141]{AtSi5} and H\"ormander~\cite[p.~153]{Hoer}.

\begin{dfn}
Let $Y$ be a space, $X$ a manifold, $m\in \R$, and
$E^0, E^1 \to X\times Y$ Hermitian vector bundles.
The set $\PsiDO^m_Y(X;E^0,E^1)$ of \emph{$Y$\nobreakdash-families of $m$\nobreakdash-th order
pseudo-differential operators over $X$} consists of families of $\C$\nobreakdash-linear maps
\[
	P_y \colon C^\infty_\mathrm{cpt}(X\times \{y\}, E^0)
	\longrightarrow
	C^\infty(X\times \{y\}, E^1), \qquad y \in Y,
\]
having a Schwartz kernel of prescribed local form as follows. Let
$x\colon U \to \Omega \subset \R^d$ be a coordinate on an open subset $U \subset X$
(possibly disconnected), $f \in C^\infty_\mathrm{cpt}(U)$, $V \subset Y$ open, and
$\tau^0$ and $\tau^1$ unitary trivializations of $E^0|_{U\times V}$ and $E^1|_{U\times V}$.
Given this, there should exist a $V$\nobreakdash-family of \emph{total symbols}, homomorphisms
\[
	p_y^{(f)}\colon U\times \R^d \longrightarrow \Hom_\C(E^0|_{U\times \{y\}}, E^1|_{U\times \{y\}}),\qquad
	\forall y\in V,
\]
such that for all $s \in C^\infty_\mathrm{cpt}(U\times \{y\},E^0)$ we can write
\e\label{equation.2.3}
	P_y(fs)(x)
	=
	(2\pi)^{-n}\int_{\R^d} p_y^{(f)}(x,\xi)\hat{s}(\xi)e^{i\langle x,\xi \rangle}d\xi,
	\qquad
	\forall x\in U.
\e
(The left hand side need not be supported in $U$.)
Here the integral and Fourier transform are performed on $\Omega$
using $x, \tau^0, \tau^1$.
%in the given coordinates and bundle trivializations. 
%For the total symbols
The following conditions are required in order for the
oscillatory integral \eqref{equation.2.3} to be well-behaved:
\begin{enumerate}
\item[(i)]
$\forall \al, \be \in \N^d, y \in V$ we have finite bounds
\[
	\|p_y^{(f)}\|_{\al,\be} \coloneqq
	\sup_{x\in U, \xi \in \R^d}
	\frac{
	\left\|\frac{\partial^{|\al|}}{\partial x^\al}\frac{\partial^{|\be|}}{\partial \xi^\be} p_y^{(f)}(x,\xi) \right\|
	}
	{
	{(1+\|\xi\|)^{m-\be}}
	} < +\infty.
\]
\item[(ii)]
% Since $Y$ is metrizable, it is equivalent to require this only for compact subspaces $Y_0$ of $Y$.
$\forall \al, \be \in \N^d, y_0 \in V, \varepsilon > 0$
there exists a neighborhood $V_0 \subset V$ of $y_0$ with
\e\label{continuity-of-PSIDO}
	\|p_y^{(f)}-p_{y_0}^{(f)}\|_{\al,\be} \leq \varepsilon,
	\qquad\forall y\in V_0.
\e
\item[(iii)]
The limit
$\sigma_{x,\xi}^{(f)}(P_y) \coloneqq \lim_{\lambda \to +\infty} \lambda^{-m}\cdot p^{(f)}_y(x,\lambda\cdot \xi)$ exists.
% Stronger reality condition:
%\item[(iv)]
% $\overline{p^{(f)}_y(x,\xi)} = p^{(f)}_y(x,-\xi)$.
\end{enumerate}
Since $p_y^{(f)}$ is uniquely determined by $P$, $f$ and the trivializations $x, \tau$,
we can define
the \emph{principal symbol family} $\sigma(P)\colon \pi^*E^0 \to \pi^*E^1$
by $\sigma_\xi(P_y) \coloneqq \sigma_{x,\xi}^{(f)}(P_y)$
for any $\xi \in T^*_x U$ and $f \in C^\infty_\mathrm{cpt}(U)$
with $f\equiv 1$ near $x$.

We consider the following \emph{types} of $P \in \PsiDO_Y^m(X;E^0,E^1)$:
{\ITEMorient} $P$ is \emph{real} if $E^0$ and $E^1$ are equipped with
orthogonal real structures and
$P_y(\overline{s})=\overline{P_y(s)}$,
{\ITEMpfaff}
$P$ is real skew-adjoint if $P$ is real and $P^\dagger=-P$ for the formally adjoint operator,
and {\ITEMspectral}
$P$ is self-adjoint if $P^\dagger = P$.
\end{dfn}

%\begin{dfn}
%\hangindent\leftmargini
%(a)\hskip\labelsep
%$P \in \PsiDO_Y^m(X;E^0,E^1)$ is \emph{real} if $E^0$ and $E^1$ are equipped with
%real structures (complex anti-linear involutions `$\overline{\,\cdot\,}$' preserving the metric) and
%$P_y(\overline{s})=\overline{P_y(s)}$ for all $s \in C^\infty_\mathrm{cpt}(X\times \{y\}, E^0)$.
%\begin{enumerate}
%\item[\textup{(b)}]
%The (formal) \emph{adjoint} of
%$P \in \PsiDO_Y^m(X;E^0,E^1)$ is the unique pseudo-differential operator family $P^\dagger$ satisfying
%$\langle P_y(s), t \rangle=\langle s, P^\dagger_y(t) \rangle$ for all $s \in C^\infty_\mathrm{cpt}(X\times \{y\}, E^0)$, $t \in C^\infty_\mathrm{cpt}(X\times \{y\}, E^1)$, $y \in Y$.
%\end{enumerate}
%\end{dfn}

\subsection{Compactly supported operators}

\begin{dfn}
A zeroth-order family $P \in \PsiDO^0_Y(X;E^0,E^1)$ 
is \emph{compactly supported} in an open ${U \subset X}$ if
there exists ${\phi \in C^\infty_\mathrm{cpt}(U, [0,1])}$
and a $Y$\nobreakdash-family of bundle homomorphisms
${\widetilde{p} \colon E^0|_{\left(X \setminus L\right) \times Y} \to E^1|_{\left(X \setminus L\right) \times Y}}$, where ${L \coloneqq \phi^{-1}(1)}$,
satisfying the following conditions:
\begin{gather}
\forall y\in Y, s\in C^\infty_\mathrm{cpt}(X\times \{y\}, E^0):
P_y(s) = \phi P_y(\phi s) + (1-\phi^2) \widetilde{p}_y\circ s.\label{eqn:P-cpt-supported}\\
\forall y \in Y\mkern2mu \exists C > 0 \forall x \notin L, y \in Y : \|\widetilde{p}_{x,y}\|\leq C.\label{eqn:P-cpt-supported-bounds}\\
\begin{gathered}
\text{$\forall \varepsilon > 0, y_0 \in Y$ there exists a neighborhood $V$ of $y_0$ such that}\\
(x,y) \in (X \setminus L) \times V
\implies
\|\widetilde{p}_{x,y} - \widetilde{p}_{x,y_0}\| \leq \varepsilon.
\end{gathered}\label{eqn:P-cpt-supported-continuous}
\end{gather}
\end{dfn}

\begin{lem}\label{lem:convex}
Assume $\psi \in C^\infty(X)$ satisfies $\psi|_{\supp \phi} \equiv 1$. Then
\eqref{eqn:P-cpt-supported} holds with $\psi$ in place of $\phi$. In particular, a finite convex combination of elements in $\PsiDO_Y^0(X; E^0, E^1)$ compactly supported in $U$ is again compactly supported in $U$.
\end{lem}

\begin{proof}
By assumption, $\phi \psi = \phi$. Then
\begin{align*}
	\psi P_y(\psi s) + (1-\psi^2)\widetilde{p}_y\circ s
	&\overset{\eqref{eqn:P-cpt-supported}}{=}
	\psi \left(
	\phi P_y(\phi \psi s) + (1-\phi^2)\widetilde{p}_y\circ \psi s
	\right)
	+(1-\psi^2)\widetilde{p}_y\circ s\\
	&=
	\phi P_y(\phi s) + (1-\phi^2)\widetilde{p}_y\circ s = P_y(s).
\end{align*}
In a finite convex combination we may therefore use a common cut-off $\phi$ for each
of the operators to check condition \eqref{eqn:P-cpt-supported}.
\end{proof}

Pseudo-differential operators are not local in general, but
assuming \eqref{eqn:P-cpt-supported} we can at least restrict to $U$.

\begin{dfn}\label{dfn:pullback-family}
Let $i\colon U \to X$ be an open embedding of 
manifolds, not necessarily compact and
$P \in \PsiDO_Y^0(X;E^0,E^1)$ compactly supported in $i(U)$.
For $y \in Y$ and $s\in C^\infty_\mathrm{cpt}(U\times \{y\}, E^0)$
the section $P_y(s\circ i^{-1})$ is supported in $i(U)$, by
\eqref{eqn:P-cpt-supported}. The
\emph{pullback family} in $\PsiDO_Y^0(U; i^*E^0, i^*E^1)$ is
\[
	\left(i^*P\right)_y(s)\coloneqq P_y(s\circ i^{-1})\circ i.
\]
% Still compactly supported in $U$.
Given, in addition, a pair of unitary bundle isomorphisms $\Phi^\bu \colon F^\bu \to i^*E^\bu$,
we may define $(i,\Phi)^*P_y(s) \coloneqq (\Phi^1)^{-1}\circ P_y(\Phi^0\circ s\circ i^{-1})\circ i$ acting on sections of $F^\bu$.
\end{dfn}

\begin{dfn}
Let $P \in \PsiDO^0_Y(X; E^0, E^1)$.
Then $P$ is \emph{compactly supported elliptic in $U \subset X$}
if $P$ is compactly supported in $U$ and $\sigma(P)$ is
an elliptic symbol family supported in a compact subset $L\times Y \subset U$.
%For $P \in \PsiDO^m_Y(X; E^0, E^1)$ and $m\neq 0$ we define
%ellipticity only for $X$ compact to mean $\sigma(P)$ is an
%elliptic symbol family.
\end{dfn}

\subsection{Fredholm results}

We next summarize the main properties of compactly supported families.
These are well-known for $X$ and $Y$ compact, see Atiyah--Singer~\cite{AtSi4}.

\begin{thm}\label{thm:Fredholm-Properties}
Let $Y$ be space, $X$ an oriented Riemannian manifold, and let
$E^0, E^1 \to X \times Y$ be Hermitian vector bundles. Let
$L^2(X,E^0), L^2(X,E^1)$ be the Hilbert space bundles over
$Y$ of $L^2$\nobreakdash-sections in $X$\nobreakdash-direction.
\begin{enumerate}
\item[\textup{(i)}]
Let $P \in \PsiDO^0_Y(X; E^0,E^1)$ be a compactly supported $Y$\nobreakdash-family.
Then $P$ can be extended to a continuous bundle map % over the identity on $Y$
\e\label{equation.2.4}
	P\colon L^2(X,E^0) \longrightarrow L^2(X, E^1)
\e
restricting fiberwise to bounded operators.
	\item[\textup{(ii)}]
	Let $P \in \PsiDO^0_Y(X; E^0,E^1)$ be compactly supported elliptic in $U$.
	Then each $\Ker P_y$ is finite-dimensional and contained in $C^\infty_\mathrm{cpt}(U,E^0)$.
	The operator $P^\dagger$ is also compactly supported elliptic in $U$.
	Hence the restrictions of \eqref{equation.2.4} to all of the fibers are
	Fredholm.

	When $E^0=E^1$, the operator $P-\lambda$ is compactly supported elliptic in
	$U$ for all $|\lambda|$ sufficiently small, depending on the constant $c$ in
	\eqref{dfn.symb-compactly-supp-bounds}.
	In particular, all eigenspaces of $P_y$ with eigenvalue $\lambda$ near
	zero are finite-dimensional and all eigenfunctions belong
	 to $C^\infty_\mathrm{cpt}(U,E^0)$.
	\item[\textup{(iii)}]
	Every family of elliptic symbols $p\in \Ell_Y^0(X;E^0,E^1)$ is the principal symbol
	of a family $P$ in $\PsiDO^0_Y(X; E^0,E^1)$.
	Moreover, a real, skew-adjoint, or self-adjoint family $p$ can be realized by a
	family $P$ of corresponding type.
	When $p$ is compactly supported in $L\subset U \subset X$ and $U$ is open,
	the operator $P$ can be chosen to be compactly supported elliptic in $U$.
	In each case there is a convex space of choices for $P$.
\end{enumerate}
\end{thm}

\begin{proof}
\textup{(i)}\hskip\labelsep
The key point is the following local estimate.
For $y_0\in Y$ and functions $f,g \in C^\infty_\mathrm{cpt}(V)$ supported in
a chart neighborhood $V$ we have
\e\label{equation.2.5}
	\|gP_{y_0}(f s)\|_{L^2} \leq D_{y_0} \cdot \|s\|_{L^2},\qquad
	\forall s \in C^\infty_\mathrm{cpt}(V\times \{y_0\},E^0).
\e
Here the constant $D_{y_0}$ is a homogeneous linear polynomial in
$\|p^{(f)}_{y_0}\|_{\al,0}$ for ${|\al|\leq \dim X+1}$ whose coefficients
depend on $\dim X$, lower and upper bounds for the density $\operatorname{vol}|_V$
with respect to the Lebesgue measure (e.g.~if $V$ is relatively compact),
and uniform bounds for the first $\dim X + 1$ derivatives
of $g$. This estimate is an adaption of H\"ormander's proof for
\cite[Th.~3.1]{Hoer} on p.~154.
\smallskip\\
\indent
Assume now $P$ is compactly supported in $U$.
Using a Lebesgue % argument on $X$, so choose a aux. metric
number we find finitely many relatively compact chart neighborhoods $\{V_a\}_{a \in A}$
covering $L$ such that for all $a,b \in A$ the union ${V_a \cup V_b}$ is also a chart. Pick $\chi_a \in C^\infty_\mathrm{cpt}(V_a)$ with $\sum_a \chi_a^2 \equiv 1$ on $L$.
For $s\in C^\infty_\mathrm{cpt}(\{y_0\}\times X, E^0)$ we find
\begin{align*}
\|P_{y_0}s\|_{L^2} &\leq \|\phi P_{y_0}(\phi s)\|_{L^2} + \|(1-\phi^2) \widetilde{p}_{y_0}\circ s\|_{L^2}
&&\text{\eqref{eqn:P-cpt-supported}}\\
& \leq \sum_{a,b} \|\chi_a^2\phi P_{y_0}(\chi_b^2\phi s)\|_{L^2} + C_{y_0}\cdot \|s\|_{L^2}
&&\text{\eqref{eqn:P-cpt-supported-bounds}}\\
&\leq \sum_{a,b} D_{y_0,a,b}\cdot\|\chi_b s\|_{L^2}+ C_{y_0}\cdot \|s\|_{L^2}
&&\text{\eqref{equation.2.5}}\\ % take f=\chi_a^2\phi, g=\chi_b\phi.
&\leq \mathrm{const}\cdot \|s\|_{L^2}.
\end{align*}
Hence $P_{y_0}$ extends to a bounded operator on $L^2$\nobreakdash-sections.
% A function on a second-countable space is continuous iff it is sequencially
% continuous. A convergent sequence with its limit point is a compact subset.
To prove continuity we can use the same estimates for $\|P_ys - P_{y_0}s\|_{L^2}$ as above, now using \eqref{continuity-of-PSIDO} and \eqref{eqn:P-cpt-supported-continuous} to
see that the constants tend to zero for $y$ in a neighborhood of $y_0$.

\smallskip\noindent
\textup{(ii)}\hskip\labelsep
A convenient parametrix $Q_{y_0}$ for $P_{y_0}$ is obtained by patching local parametrices on the charts $U_a$ with the \emph{exact} inverse $\widetilde{p}_{y_0}^{-1}$ outside $L$,
which is a bounded operator by the lower bound of
\eqref{dfn.symb-compactly-supp-bounds}.

Then we find ${Q_{y_0}P_{y_0} - 1 = R_1}$, ${P_{y_0}Q_{y_0} - 1 = R_2}$ with
$R_1, R_2$ having compactly supported $C^\infty$\nobreakdash-kernels. These define compact operators on $L^2$.
The remaining assertions are easily verified.

\smallskip\noindent
\textup{(iii)}\hskip\labelsep
This is true locally, since we can define $P$ by \eqref{equation.2.3}
for $p^{(f)}_y(x,\xi) \coloneqq f(x)p_{y,\xi}$,
and then patched to a global result using a partition of unity on $X$.
When the elliptic symbol family $p$ is supported in $L\subset U \subset X$,
by definition we have $\widetilde{p}$ outside $L$.
Pick $\phi \in C^\infty_\mathrm{cpt}(U,[0,1])$
with $\phi|_L \equiv 1$. Then, beginning with an arbitrary $Y$\nobreakdash-family $P$ in $\PsiDO_Y^0(X; E^0,E^1)$ with principal symbol family $p$, replace it by ${\phi_y P_y(\phi s) + (1-\phi^2)\widetilde{p}_y\circ s}$ to get one that is compactly supported in $U$.
The reality condition can be ensured by passing to $\frac12(\overline{P_y(s)} + P_y(\overline{s})$
and similarly $\frac12(P_y(s) \pm P^\dagger_y(s)$ ensures skew-adjointness
or self-adjointness.
Finally, a straight-line interpolation between two
families in $\PsiDO_Y^0(X; E^0,E^1)$ with given principal symbols remains a family 
in $\PsiDO_Y^0(X; E^0,E^1)$ with that symbol. In the compactly supported case we use
Lemma~\ref{lem:convex} here.
\end{proof}

\medskip

\noindent{\small\sc The Mathematical Institute, Radcliffe
Observatory Quarter, Woodstock Road, Oxford, OX2 6GG, U.K.

\noindent E-mail: {\tt upmeier@maths.ox.ac.uk.}}


\begin{thebibliography}{99}
\addcontentsline{toc}{section}{References}

%\bibitem{Atiy} M.F. Atiyah, {\it K-Theory}, W.A. Benjamin inc., 1967 / Addison-Wesley, 1989.
%
%\bibitem{AtBo} M.F. Atiyah and R. Bott, {\it The Yang--Mills equations over Riemann surfaces}, Phil. Trans. Roy. Soc. London A308 (1982), 523--615.
%
%\bibitem{AtBoSh} M.F. Atiyah, R. Bott and A. Shapiro, {\it Clifford modules}, Topology 3 (1964), 3--38.
%
%\bibitem{AtHi} M.F. Atiyah and F. Hirzebruch, {\it The Riemann-Roch theorem for analytic embeddings}, 
%Topology 1 (1962), 151--166.
%
\bibitem{AtSiSkew} M.F. Atiyah and I.M. Singer, {\it Index theory for skew-adjoint {F}redholm operators},
Inst. Hautes \'{E}tudes Sci. Publ. Math. 37 (1969), 5--26.

\bibitem{AtPaSi1} M.F. Atiyah, V.K. Patodi and I.M. Singer, {\it Spectral asymmetry and Riemannian geometry. I}, Math. Proc. Cambridge Philos. Soc. 77 (1975), 43--69.
%
%\bibitem{AtPaSi2} M.F. Atiyah, V.K. Patodi and I.M. Singer, {\it Spectral asymmetry and Riemannian geometry. II}, Math. Proc. Cambridge Philos. Soc. 78 (1975), 405--432.

\bibitem{AtPaSi3} M.F. Atiyah, V.K. Patodi and I.M. Singer, {\it Spectral asymmetry and Riemannian geometry. III}, Math. Proc. Cambridge Philos. Soc. 79 (1976), 71--99.

\bibitem{AtSi1} M.F. Atiyah and I.M. Singer, {\it The Index of Elliptic Operators: I}, Ann. of Math. 87 (1968), 484--530.

%\bibitem{AtSe} M.F. Atiyah and G.B. Segal, {\it The Index of Elliptic Operators: II}, Ann. of Math. 87 (1968), 531--545.

%\bibitem{AtSi3} M.F. Atiyah and I.M. Singer, {\it The Index of Elliptic Operators: III}, Ann. of Math. 87 (1968), 546--604.

\bibitem{AtSi4} M.F. Atiyah and I.M. Singer, {\it The Index of Elliptic Operators: IV}, Ann. of Math. 92 (1970), 119--138.

\bibitem{AtSi5} M.F. Atiyah and I.M. Singer, {\it The Index of Elliptic Operators: V}, Ann. of Math. 93 (1971), 139--149.


%\bibitem{BeFa} K. Behrend and B. Fantechi, {\it The intrinsic normal cone}, Invent. Math. 128 (1997), 45--88. \href{http://arxiv.org/abs/alg-geom/9601010}{alg-geom/9601010}.

\bibitem{BiFr} J.M. Bismut and D.S. Freed, {\it The analysis of elliptic families. I. Metrics and connections on determinant bundles}, Comm. Math. Phys. 106 (1986), 159--176.

%\bibitem{Blan} A. Blanc, {\it Topological K-theory of complex noncommutative spaces}, Compositio Mathematica 152 (2015), 489--555. \href{http://arxiv.org/abs/1211.7360}{arXiv:1211.7360}.

%\bibitem{BoJo} D. Borisov and D. Joyce, {\it `Virtual fundamental classes for moduli spaces of sheaves on Calabi--Yau four-folds'}, Geometry and Topology 21 (2017), 3231--3311. \href{http://arxiv.org/abs/1504.00690}{arXiv:1504.00690}.

%\bibitem{Bott} R. Bott, {\it Quelques remarques sur les th\'{e}or\`{e}mes de p\'{e}riodicit\'{e}},
%Bull. Soc. Math. France 87 (1959), 293--310.

\bibitem{CaJo} Y. Cao and D. Joyce, {\it Orientability of moduli spaces of $\Spin(7)$-instantons},
\href{http://arxiv.org/abs/1811.09658}{arXiv:1811.09658}, 2018.

\bibitem{CaGrJo} Y. Cao, J. Gross and D. Joyce, {\it On orientations for moduli spaces of coherent
sheaves on Calabi--Yau manifolds}, preprint, 2019.

%\bibitem{CaLe1} Y. Cao and N.C. Leung, {\it Donaldson--Thomas theory for Calabi--Yau four-folds}, \href{http://arxiv.org/abs/1407.7659}{arXiv:1407.7659}, 2014.

%\bibitem{CaLe2} Y. Cao and N.C. Leung, {\it Orientability for gauge theories on Calabi--Yau manifolds}, Adv. Math. 314 (2017), 48--70. \href{http://arxiv.org/abs/1502.01141}{arXiv:1502.01141}.

%\bibitem{CrNo} D. Crowley and J. Nordstr\"om, {\it New invariants of\/ $G_2$-structures}, Geom. Topol. 19 (2015), 2949--2992. \href{http://arxiv.org/abs/1211.0269}{arXiv:1211.0269}.

%\bibitem{deJo} A.J. de Jong et al., {\it The Stacks Project}, online textbook available at \href{https://stacks.math.columbia.edu}{\tt https://stacks.math.columbia.edu}, 2005--.

\bibitem{Del} P. Deligne, {\it Le d\'{e}terminant de la cohomologie},
pages 93--177 in {\it Current trends in arithmetical algebraic geometry \textup({A}rcata, {C}alif., 1985\textup)},
Contemp. Math. 67, AMS, Providence, R.I., 1987.

%\bibitem{Dold} A. Dold, {\it Lectures on algebraic topology}, Grundlehren der math. Wiss. 200, Springer-Verlag, Berlin--New York, 1980.

%\bibitem{Dona1} S.K. Donaldson, {\it An application of gauge theory to four-dimensional topology}, J. Diff. Geom. 18 (1983), 279--315.

\bibitem{Dona2} S.K. Donaldson, {\it The orientation of Yang--Mills moduli spaces and\/ $4$-manifold topology}, J. Diff. Geom. 26 (1987), 397--428. 

\bibitem{DoKr} S.K. Donaldson and P.B. Kronheimer, {\it The Geometry of Four-Manifolds}, OUP, 1990.

\bibitem{DoSe} S.K. Donaldson and E. Segal, {\it Gauge Theory in Higher Dimensions, II}, Surveys in Diff. Geom. 16 (2011), 1--41. \href{http://arxiv.org/abs/0902.3239}{arXiv:0902.3239}.

%\bibitem{DoTh} S.K. Donaldson and R.P. Thomas, {\it Gauge Theory in Higher Dimensions}, Chapter 3 in S.A. Huggett, L.J. Mason, K.P. Tod, S.T. Tsou and N.M.J. Woodhouse, editors, {\it The Geometric Universe}, Oxford University Press, Oxford, 1998.

%\bibitem{Dugu} J. Dugundji, {\it Topology}, Allyn and Bacon, Inc., Boston, Mass., 1966.

%\bibitem{Ebert} J. Ebert, {\it Elliptic regularity for Dirac operators on families of noncompact manifolds}, 
%\href{https://arxiv.org/abs/1608.01699}{arXiv:1608.01699}, 2016.

\bibitem{Floe} A. Floer, {\it Morse theory for {L}agrangian intersections}, J. Differential Geom. 28 (1988), 513--547.

\bibitem{Fre} D.S. Freed, {\it On determinant line bundles}, pages 189--238 in {\it Mathematical Aspects of String Theory}, Adv. Ser. Math. Phys. 1, World Sci. Publishing, Singapore, 1987.

%\bibitem{FHT} D.S. Freed, M.J. Hopkins, and C. Teleman, {\it Consistent orientation of moduli spaces}, pages 395--419 in {\it The many facets of geometry}, Oxford Univ. Press, Oxford, 2010. \href{http://arxiv.org/abs/0711.1909}{arXiv:0711.1909}.

%\bibitem{FU} D.S. Freed and K.K. Uhlenbeck,
%{\it Instantons and four-manifolds},
%Mathematical Sciences Research Institute Publications,
%Springer-Verlag, New York, 1984.

%\bibitem{FOOO} K. Fukaya, Y.-G. Oh, H. Ohta and K. Ono,
%{\it Lagrangian intersection Floer theory --- anomaly and
%obstruction}, Parts I \& II. AMS/IP Studies in Advanced Mathematics,
%46.1 \& 46.2, A.M.S./International Press, 2009.

%\bibitem{Gome} T.L. G\'omez, {\it Algebraic stacks}, Proc. Indian Acad. Sci. Math. Sci. 111 (2001), 1--31. \href{http://arxiv.org/abs/math/9911199}{math.AG/9911199}.
%
%\bibitem{GrHa} P. Griffiths and J. Harris, {\it Principles of Algebraic Geometry}, Wiley, New York, 1978.
%
%\bibitem{Haef} A. Haefliger, {\it Plongements diff\'erentiables de vari\'et\'es dans vari\'et\'es},
%Comment. Math. Helv. 36 (1961), 47--82.
%
%\bibitem{Hirsch} M.W. Hirsch, {\it Differential topology},
%Graduate Texts in Mathematics 33, Springer Verlag, New York, 1994.

\bibitem{Hoer} L. H\"ormander, {\it Pseudo-differential operators and hypoelliptic equations}, pages 138--183 in {\it Singular integrals}, Proc. Sympos. Pure Math. Vol. 10,
A.M.S., Providence, RI, 1967.

%\bibitem{Illu1} L. Illusie, {\it Complexe cotangent et
%d\'eformations. I}, Springer Lecture Notes in Math. 239,
%Springer-Verlag, Berlin, 1971.
%
%\bibitem{Illu2} L. Illusie, {\it Cotangent complex and deformations
%of torsors and group schemes}, pages 159--189 in Springer Lecture
%Notes in Math. 274, Springer-Verlag, Berlin, 1972.
%

%\bibitem{Joyc1} D. Joyce, {\it Conjectures on counting associative 3-folds in $G_2$-manifolds},
%pages 97--160 in {\it Modern Geometry: A Celebration of the Work of Simon Donaldson},
%Proc. Sympos. Pure Math. Vol. 99, A.M.S., Providence, RI, 2018. \href{http://arxiv.org/abs/1610.09836}{arXiv:1610.09836}

%\bibitem{Joyc2} D. Joyce, {\it Ringel--Hall style Lie algebra structures on the homology of moduli spaces}, preprint, 2018.

\bibitem{JTU} D. Joyce, Y. Tanaka and M. Upmeier, {\it  On orientations for gauge-theoretic moduli spaces}, \href{http://arxiv.org/abs/1811.01096}{arXiv:1811.01096}, 2018.

\bibitem{JoUp} D. Joyce and M. Upmeier, {\it Canonical orientations for moduli spaces of\/ $G_2$-instantons with gauge group $\SU(m)$ or $\U(m)$}, \href{http://arxiv.org/abs/1811.02405}{arXiv:1811.02405}, 2018.

%\bibitem{Karo} M. Karoubi, {\it K-Theory, an introduction}, Grundlehren der math. Wiss. 226, Springer-Verlag, Berlin, 1978.

%\bibitem{KnMu} F.F. Knudsen and D. Mumford, {\it The projectivity of the moduli space of stable curves. I. %Preliminaries on ``det'' and ``Div''}, Math. Scand. 39 (1976), 19--55. 

\bibitem{Knu} F.F. Knudsen, {\it Determinant functors on exact categories and their extensions
              to categories of bounded complexes}, Michigan Math. J. 2 (2002), 407--444.

%\bibitem{KoSo1} M. Kontsevich and Y. Soibelman, {\it Stability structures, motivic Donaldson--Thomas invariants and cluster transformations}, \href{http://arxiv.org/abs/0811.2435}{arXiv:0811.2435}, 2008.

%\bibitem{KoSo2} M. Kontsevich and Y. Soibelman, {\it Motivic Donaldson--Thomas invariants: summary of results}, pages 55--89, in {\it Mirror symmetry and tropical geometry}, Contemp. Math. 527, A.M.S., Providence, RI, 2010. \href{http://arxiv.org/abs/0910.4315}{arXiv:0910.4315}.

%\bibitem{LaMo} G. Laumon and L. Moret-Bailly, {\it Champs alg\'ebriques}, Ergeb. der Math. und ihrer Grenzgebiete 39, Springer-Verlag, Berlin, 2000.

\bibitem{LaMi} H.B. Lawson and M.-L. Michelsohn, {\it Spin geometry}, Princeton Math. Series 38,
Princeton Univ. Press, Princeton, NJ, 1989.

%\bibitem{MNS} G. Menet, J. Nordstr\"om and H.N. S\'a Earp, {\it Construction of\/ $G_2$-instantons via twisted connected sums}, \href{http://arxiv.org/abs/1510.03836}{arXiv:1510.03836}, 2015.

%\bibitem{Metz} D.S. Metzler, {\it Topological and smooth stacks}, \href{http://arxiv.org/abs/math/0306176}{math.DG/0306176}, 2003.

%\bibitem{MiSt} J.W. Milnor and J.D. Stasheff, {\it Characteristic classes}, Princeton University Press, Princeton, NJ, 1974.

%\bibitem{Moch} T. Mochizuki, {\it Donaldson type invariants for algebraic surfaces}, Lecture Notes in Math. 1972, Springer, Berlin, 2009.

%\bibitem{MuSh} V. Mu\~noz and C.S. Shahbazi, {\it Orientability of the moduli space of\/ $\Spin(7)$-instantons}, \href{http://arxiv.org/abs/1707.02998}{arXiv:1707.02998}, 2017.

%\bibitem{Nooh1} B. Noohi, {\it Foundations of topological stacks, I}, \href{http://arxiv.org/abs/math/0503247}{math.AG/0503247}, 2005.
%
%\bibitem{Nooh2} B. Noohi, {\it Homotopy types of topological stacks}, \href{http://arxiv.org/abs/0808.3799}{arXiv:0808.3799}, 2008.
%
%\bibitem{Olss1} M. Olsson, {\it Algebraic Spaces and Stacks}, A.M.S. Colloquium Publications 62, A.M.S., Providence, RI, 2016.
%
%\bibitem{ReCa} R. Reyes Carri\'on, {\it A generalization of the notion of instanton}, Differential Geom. Appl. 8 (1998), 1--20.
%
%\bibitem{SaEa} H.N. S\'a Earp, {\it $G_2$-instantons over asymptotically cylindrical manifolds}, Geom. Topol. 19 (2014), 61--111. \href{http://arxiv.org/abs/1101.0880}{arXiv:1101.0880}.
%
%\bibitem{SEWa} H.N. S\'a Earp and T. Walpuski,  {\it $G_2$-instantons over twisted connected sums},  Geom. Topol. 19 (2015), 1263--1285. \href{http://arxiv.org/abs/1310.7933}{arXiv:1310.7933}.

\bibitem{Phi} J. Phillips, {\it Self-adjoint Fredholm operators and spectral flow},
Canad. Math. Bull. 39 (1996), 460--467.

\bibitem{Seel} R.T. Seeley, {\it Integro-differential operators on vector bundles},
Trans. Amer. Math. Soc. 117 (1965), 167--204.

%\bibitem{Swit} R.M. Switzer, {\it Algebraic Topology -- Homotopy and Homology}, Grundlehren der math. Wiss. 212, Springer-Verlag, New York, 1975.

%\bibitem{Wall} C.T.C. Wall, {\it All {$3$}-manifolds imbed in {$5$}-space}, Bull. Amer. Math. Soc. 71 (1965), 564--567.

%\bibitem{Walp1} T. Walpuski, {\it $G_2$-instantons on generalized Kummer constructions},  Geom. Topol. 17 (2013), 2345--2388. \href{http://arxiv.org/abs/1109.6609}{arXiv:1109.6609}.

%\bibitem{Walp2} T. Walpuski, {\it $G_2$-instantons, associative submanifolds and Fueter sections}, \href{http://arxiv.org/abs/1205.5350}{arXiv:1205.5350}, 2012.

\bibitem{Walp3} T. Walpuski, {\it Gauge theory on $G_2$-manifolds}, PhD Thesis, Imperial College London, 2013.

%\bibitem{Walp4} T. Walpuski, {\it $G_2$-instantons over twisted connected sums: an example}, Math. Res. Lett. 23 (2016), 529--544. \href{http://arxiv.org/abs/1505.01080}{arXiv:1505.01080}.

%\bibitem{Zent} R. Zentner, {\it On higher rank instantons and the monopole cobordism program}, Q. J. Math. 63 (2012), 227--256. \href{http://arxiv.org/abs/0911.5146}{arXiv:0911.5146}.

\bibitem{Zing} A. Zinger, {\it The determinant line bundle for {F}redholm operators:
	construction, properties, and classification}, Math. Scand. 118 (2016), 203--268. \href{http://arxiv.org/abs/1304.6368}{arXiv:1304.6368}

\end{thebibliography}
\end{document}